\documentclass{amsart}
\usepackage{amsmath,amssymb,amscd,latexsym}

\usepackage[all]{xy}

\usepackage{graphicx}

\usepackage{mathrsfs}

\usepackage{stmaryrd}

\usepackage{enumitem}
\usepackage{nicefrac}
\usepackage{multirow}

\usepackage{tikz}
\usepackage{tikz-cd}

\usepackage{hyperref}

\usepackage[normalem]{ulem}

\definecolor{NTNUblue}{RGB}{0,80,158}
\definecolor{NTNUbluesupport}{RGB}{62,98,138}
\definecolor{NTNUorange}{RGB}{239,129,20}

\newcommand*\circled[1]{\tikz[baseline=(char.base)]{
    \node[shape=circle,draw,inner sep=2pt] (char) {#1};}}

\newcommand{\bi}{\begin{itemize}}
\newcommand{\ei}{\end{itemize}}
\newcommand{\bn}{\begin{enumerate}}
\newcommand{\en}{\end{enumerate}}

\newcommand{\bq}{\begin{equation}}
\newcommand{\eq}{\end{equation}}

\newcommand{\ba}{\begin{align}}
\newcommand{\ea}{\end{align}}

\newcommand{\bas}{\begin{align*}}
\newcommand{\eas}{\end{align*}}

\newcommand{\bs}{\begin{split}}
\newcommand{\es}{\end{split}}

\newcommand{\RP}{{\mathbb{RP}}}

\newcommand{\R}{{\mathbb{R}}}
\newcommand{\Z}{{\mathbb{Z}}}
\newcommand{\cl}{\mathrm{cl}}

\newcommand{\modu}[3]{#1 \equiv #2 \; (\mathrm{mod} \; #3)}
\newcommand{\fmodu}[3]{\mathbf{#1 \equiv #2} \; (\mathrm{\textbf{mod}} \; \mathbf{#3)}}
\newcommand{\nmodu}[3]{#1 \not\equiv #2 \; (\mathrm{mod} \; #3)}

\newcommand{\Gr}[2]{\widetilde{\mathrm{Gr}_{#1}}(\mathbb{R}^{#2})}

\newcommand{\flip}{\varepsilon}

\newcommand{\Span}[1]{\mathrm{Span\{{#1}\}}}

\newcommand{\Hfree}[2]{H_\text{free}^{#1}(#2;\mathbb{Z})}
\newcommand{\Imm}{\mathrm{Im}\,}
\newcommand{\Ker}{\mathrm{Ker}\,}

\newcommand{\Spin}{\mathrm{Spin}}
\newcommand{\Sq}{\mathrm{Sq}}

\newcommand{\into}{\hookrightarrow}

\newcommand{\rz}{\mathbb{R}/\mathbb{Z}}

\newcommand{\trz}{\tau_{\rz}}

\newcommand{\btrz}{\bar{\tau}_{\rz}}

\newtheorem{theorem}{Theorem}[section]
\newtheorem{lemma}[theorem]{Lemma}
\newtheorem{prop}[theorem]{Proposition}
\newtheorem{cor}[theorem]{Corollary}

\newtheorem*{theorem*}{Theorem}

\theoremstyle{definition}
\newtheorem{defn}[theorem]{Definition}

\newtheorem{remark}[theorem]{Remark}

\usepackage{adjustbox}

\begin{document}

\author{Eiolf Kaspersen}
\address{Department of Mathematical Sciences, NTNU, NO-7491 Trondheim, Norway}
\email{eiolf.kaspersen@ntnu.no}

\author{Gereon Quick} 
\address{Department of Mathematical Sciences, NTNU, NO-7491 Trondheim, Norway}
\email{gereon.quick@ntnu.no}

\title{On the cokernel of the Thom morphism for compact Lie groups}

\date{}

\begin{abstract}
We give a complete description of the potential failure of the surjectivity of the Thom morphism from complex cobordism to integral cohomology for compact Lie groups via a detailed study of the  Atiyah--Hirzebruch spectral sequence and the action of the Steenrod algebra. 
We show how the failure of the surjectivity of the topological Thom morphism can be used to find examples of non-trivial elements in the kernel of the induced differential Thom morphism from differential cobordism of Hopkins and Singer to differential cohomology. 
These arguments are based on the particular algebraic structure and interplay of the torsion and non-torsion parts of the cohomology and cobordism rings of a given compact Lie group. 
We then use the geometry of special orthogonal groups to construct concrete cobordism classes in the non-trivial part of the kernel of the differential Thom morphism.  
\end{abstract}

\maketitle

\setcounter{tocdepth}{3}

\tableofcontents

\section{Introduction}

The Thom morphism 
\[
\tau \colon MU \longrightarrow H\Z
\]
from complex cobordism to integral singular cohomology is of fundamental importance for the study of the stable homotopy category. 
A special feature of $\tau$ is that it encodes both deep  algebraic and geometric structures. 
This is a common theme of the present paper and 
is reflected in the following two ways $\tau$ may be described.  
On the one hand, $\tau$ interpolates between two extreme ends of the spectrum of oriented cohomology theories which may be classified by their formal group laws, as $\tau$ corresponds the unique morphism from the universal formal group law to the additive one (see \cite[II Example (4.7)]{adams}). 
On the other hand, $\tau$ may be described geometrically in the following way.  
Let $X$ be a smooth manifold. 
By Quillen's work in \cite{quillen}, classes in $MU^*(X)$ can be represented by proper complex-oriented maps $g \colon M \to X$.   
The Thom morphism sends the class $[g]$ to $g_*[M]$ where $[M]$ denotes the Poincar\'e dual of the fundamental class of $M$. 
Thus, roughly speaking, a cohomology class is in the image of $\tau$ if it can be expressed by a fundamental class of an almost-complex manifold. 
Hence the question whether $\tau$ is surjective or not is directly connected to concrete geometric phenomena, which is also why Thom introduced $\tau$ to solve Steenrod's problem in \cite{thom}. 
In cohomological degrees $i=0,1,2$, the Thom morphism is surjective for all spaces, since the Eilenberg--MacLane spaces $K(\Z,i)$ are torsion-free for $i=0,1,2$. 
In cohomological degrees $i\ge 3$, however, $\tau$ may fail to be surjective, 
even though the coefficient ring of $MU$ is much larger than the one of $H\Z$. 
It is well-known that the Atiyah--Hirzebruch spectral sequence  
\begin{align*}
    E_2^{p,q} = H^p(X;MU^q) \implies MU^{p+q}(X)
\end{align*}
both provides a way to show that $\tau$ may be surjective  
and that its differentials may yield obstructions to the surjectivity of $\tau$ (see \cite{ah}). 
However, the image of the Thom morphism has not been studied for many types of spaces. \\

The purpose of the present paper is to give a complete description of the potential failure of the surjectivity of the Thom morphism for compact connected Lie groups which provide an important class of examples of smooth manifolds. 
Our first main result is the following: 

\begin{theorem}\label{thm:main_intro}
Let $G$ be a compact connected Lie group with simple Lie algebra. 
Then table \eqref{table_intro} shows the minimal cohomological degree $q$ for which the Thom morphism $\tau \colon MU^q(G) \longrightarrow H^q(G;\mathbb{Z})$ fails to be surjective. 
\end{theorem}

In fact, for each minimal cohomological degree where $\tau$ fails to be surjective, we provide concrete non-torsion classes in $H^k(G;\mathbb{Z})$ which are not in the image of $\tau$. 
The methods to prove theorem \ref{thm:main_intro} are described in sections \ref{sec:surjective_Thom} and \ref{sec:obstructions}, and the study of the individual types of Lie groups occupies section \ref{sec:Lie_groups}. 
We note that generalised cohomology groups for some types of compact Lie groups are well-known, for example for complex $K$-theory from \cite{hodgkin}, for exceptional Lie groups and Morava $K$-theory form \cite{hmns, nishimoto}, and in Brown--Peterson cohomology from \cite{yagita1, yagita2, yagita3}. 
Some of our computations could have been deduced from these papers.  
However, in order to give a unified and self-contained picture we provide direct proofs for all groups we consider.  

\begin{table}[h]\label{table_intro}
\caption{Summary of the results of theorem \ref{thm:main_intro}}
\renewcommand{\arraystretch}{1.5}
\setlength{\tabcolsep}{6pt} 
\begin{tabular}{|c|p{3.5cm}|p{2.7cm}|p{3cm}|} 
\hline
\multicolumn{1}{|l|}{Lie Algebra} & Lie Group                                            & Surjective                           & Min. degree where surjectivity fails \\ \hline
\multirow{2}{*}{$\mathfrak{a_n}$}            & $SU(n)$ - Special unitary group         & yes                                  & --    \\ \cline{2-4}
                                  & $SU(n)/\Gamma_l$ - Quotient of special unitary group & not for $4 \mid n$ and \newline $\modu{l}{2}{4}$,\newline yes otherwise & $2^r - 1$ where $r\in \Z$ is max.\,st.\,$2^r \mid n$    \\ \hline
\multirow{2}{*}{$\mathfrak{c_n}$}            & $Sp(n)$ - Symplectic group             & yes                                  & --    \\ \cline{2-4}
                                  & $PSp(n)$ - Projective symplectic group               & not for $n$ even, \newline yes for $n$ odd                  & $2^{r+1}-1$ where $r\in \Z$ is max.\,st.\,$2^r \mid n$   \\ \hline
\multirow{4}{*}{$\mathfrak{b_n}$, $\mathfrak{d_n}$}     & $Spin(n)$ - Spin group            & not for $n \geq 7$                   & 3   \\ \cline{2-4}
                                  & $SO(n)$ - Special orthogonal group                   & not for $n \geq 5$                   & 3    \\ \cline{2-4}
                                  & $Ss(n)$ - Semi-spin group                            & not for $n \geq 4$                   & $3$ if $8 \mid n$, \newline $7$ otherwise    \\ \cline{2-4}
                                  & $PSO(n)$ - Projective special orthog. group       & not for $n \geq 8$                   & $3$ if $8 \mid n$, \newline $7$ otherwise    \\ \hline
$\mathfrak{g_2}$                             & $G_2$                                                & no                                   & $3$    \\ \hline
$\mathfrak{f_4}$                             & $F_4$                                                & no                                   & $3$   \\ \hline
\multirow{2}{*}{$\mathfrak{e_6}$}            & $E_6$, simply-connected         & no                                   & $3$   \\ \cline{2-4}
                                  & $E_6 / \Gamma_3$, centerless     & no                                   & $3$   \\ \hline
\multirow{2}{*}{$\mathfrak{e_7}$}            & $E_7$, simply-connected                 & no                                   & $3$   \\ \cline{2-4}
                                  & $E_7 / \Gamma_2$, centerless                         & no                                   & $3$   \\ \hline
$\mathfrak{e_8}$                             & $E_8$                                 & no                                   & $3$   \\ \hline
\end{tabular}
\end{table}

\begin{remark}\label{rem:intro_pattern}
We recall in section \ref{sec:surjective_Thom} why $\tau$ is surjective whenever $H^*(G;\Z)$ is torsion-free. 
However, we point out that this argument is not sufficient to explain the cases in table \eqref{table_intro} where $\tau$ is surjective. 
The pattern we observe in  
table \eqref{table_intro} indicates that Lie groups of type $\mathfrak{a_n}$ and $\mathfrak{c_n}$ tend to have a surjective Thom morphism, 
while groups of type $\mathfrak{b_n}$ and $\mathfrak{d_n}$ do not have a 
surjective Thom morphism in sufficiently high dimensions. 
The exceptional Lie groups on the other hand show a clear pattern. 
We note, however, that the behaviors of $E_7$ and $E_8$ are slightly different from the one of the other groups (see section \ref{sec:exceptional_groups}). 
We do not know of a general geometric explanation for why $\tau$ is surjective or not surjective for a given Lie group. 
In section \ref{sec:geometric_examples}, however, we use the geometry and cell structure of special orthogonal groups to construct concrete geometric elements in $MU^*(SO(n))$. 
\end{remark}


\begin{remark}\label{rem:image_of_ktheory_intro}
We note that in the cases where $\tau$ fails to be surjective in cohomological degree $3$, the generator $e_3 \in H^3(G;\Z)$ which is not hit by $\tau$ is not in the image of the homomorphism 
\begin{align*}
ku^3(G) \to H^3(G;\Z) 
\end{align*} 
from connective complex $K$-theory $ku$ either.  
This is due to the fact that the Milnor operation $Q_1$ and the Steenrod operation $\Sq^3$ provide obstructions which are differentials in the Atiyah--Hirzebruch spectral sequence   
\begin{align*}
E_2^{p,q} = H^p(G;ku^q) \implies ku^{p+q}(G).
\end{align*} 
This applies to several of the groups of type $\mathfrak{b_n}$ and $\mathfrak{d_n}$ and to all exceptional Lie groups (see table \eqref{table_intro} for the specific groups). 
In the other cases, however, surjectivity may not fail for $ku$ but only on a higher stage in the tower of cohomology theories $MU \to \cdots \to MU\langle 2 \rangle \to MU\langle 1 \rangle = ku \to MU\langle 0 \rangle = H\Z$. 
\end{remark}

 
A concrete motivation for our study of the Thom morphism arises from the theory of generalised differential cohomology theories for smooth manifolds developed by  Hopkins and Singer in \cite{hs}. 
For a rationally even spectrum $E$ and a smooth manifold $X$, the differential $E$-cohomology groups are denoted by $\check{E}(q)^n(X)$. 
The most interesting choice of degrees is $n=q$. 
The group $\check{E}(q)^q(X)$ then sits in several short exact sequences as described in \cite[diagram (4.57)]{hs}. 
In particular, the natural homomorphism $\check{E}(q)^q(X) \to E^q(X)$ is surjective. 
Hence the Thom morphism $\tau \colon MU \to H\Z$ induces a commutative diagram 
\begin{align*}
\xymatrix{
\check{MU}(q)^q(X) \ar[r] \ar[d]_-{\check{\tau}} & MU^q(X) \ar[d]^-{\tau} \\
\check{H}(q)^q(X) \ar[r] & H^q(X;\Z)
}    
\end{align*}
in which the horizontal maps are surjective. 
Thus, if $\tau$ is not surjective, then $\check{\tau}$ fails to be surjective as well. 
We note that Grady and Sati study in \cite{gradysatiAHSS} the surjectivity of the differential analog of the map from complex $K$-theory to cohomology using a differential version of the Atiyah--Hirzebruch spectral sequence. \\ 

However, the failure of the surjectivity of $\tau$ also allows us to find non-trivial elements in the kernel of $\check{\tau}$. 
For every rationally even spectrum $E$, $\check{E}(q)^q(X)$ sits in a short exact sequence of the form 
\begin{align*}
0 \to E^{q-1}(X)\otimes \rz \to \check{E}(q)^q(X) \to A_{E}^q(X) \to 0
\end{align*}
where the group $A_{E}^q(X)$ is defined by the following pullback square, in which $\Omega^*(X;\pi_*E\otimes \R)^q_{\cl}$ denotes closed forms on $X$ of total degree $q$:  
\begin{align*}
\xymatrix{
A_{E}^q(X) \ar[r] \ar[d] & \Omega^*(X;\pi_*E\otimes \R)^q_{\cl} \ar[d] \\
E^q(X) \ar[r] & H^q(X;\pi_*E\otimes \R).
}    
\end{align*}
%
The Thom morphism $\tau \colon MU \to H\Z$ induces a map of short exact sequences
\begin{align*}
\xymatrix{
0 \ar[r] & MU^{q-1}(X)\otimes_{\Z} \rz \ar[d]^-{\trz} \ar[r] & \check{MU}(q)^q(X) \ar[d]^-{\check{\tau}} \ar[r] & A_{MU}^q(X) \ar[d]^-{\tau_A} \ar[r] & 0 \\
0 \ar[r] & H^{q-1}(X;\Z)\otimes_{\Z} \rz \ar[r] & \check{H}(q)^q(X) \ar[r] & A_{H}^q(X) \ar[r] & 0. 
}    
\end{align*}
\begin{sloppypar}
Recall that the kernel of the Thom morphism always contains the ideal ${MU^{*<0}\cdot MU^*(M)}$ of $MU^*(X)$, since $\tau$ is a natural transformation of oriented cohomology theories. 
We therefore use the following terminology:\end{sloppypar} 
 
\begin{defn}\label{def:non_trivial_kernel}
We say that an element in the kernel of $\trz$ or $\check{\tau}$ is \emph{non-trivial} if it is not contained in the respective ideal generated by $MU^{*<0}$. 
\end{defn}

We will explain in section \ref{sec:differential_kernel} how the failure of $\tau$ to be surjective enables us to find non-trivial elements in the kernel of $\trz$. 
This leads to the following result, for which we emphasise that the assumption applies to a large class of compact Lie groups by theorem \ref{thm:main_intro}: 

\begin{theorem}\label{thm:kernel_diff_cob_intro}
Let $G$ be a compact Lie group $G$ and $q$ an integer such that the Thom morphism $\tau \colon MU^{q-1}(G) \longrightarrow H^{q-1}(G;\mathbb{Z})$ fails to be surjective on a non-torsion class. 
Then the kernel of the differential Thom morphism 
\begin{align*}
\check{\tau} \colon \check{MU}(q)^q(G) \to \check{H}(q)^q(G) 
\end{align*}
is non-trivial in the sense of definition \ref{def:non_trivial_kernel}.  
\end{theorem}

The significance of theorem \ref{thm:kernel_diff_cob_intro} is that, together with theorem \ref{thm:main_intro}, it provides important examples of classes on smooth manifolds which can be studied using differential cobordism but not using differential cohomology. 
We thus demonstrate by concrete examples that the generalized differential invariants of \cite{hs} are stronger than invariants that can be obtained by just using differential cohomology. 
In section \ref{sec:detecting_kernel} we explain how we can use the Atiyah--Hirzebruch spectral sequence to find non-trivial elements in the kernel of $\trz$ and $\check{\tau}$ whenever $\tau$ is not surjective.

In section \ref{sec:geometric_examples} we switch perspectives and give a concrete and geometric construction of a non-trivial element in the kernel of $\check{\tau}$ for special orthogonal groups.  
From proposition \ref{prop:SOn_for_n_ge_5} we know that 
the generator $e_3 \in H^3(SO(5);\Z)$ is not hit by $\tau$. 
In section \ref{sec:geometric_example_SO5} we show that the class $2e_3$, however, is in the image of $\tau$ by  constructing a proper complex-oriented smooth map 
\begin{align*}
g \colon \Gr{2}{5} \times S^1 \longrightarrow SO(5) 
\end{align*}
such that $\tau([g]) = 2e_3$  
where $\Gr{2}{5}$ denotes the Grassmannian of \emph{oriented} $2$-planes in $\R^5$. 
In section \ref{sec:geometric_example_SO5} we prove the following result which we generalise in section \ref{sec:generalize_to_higher_SOn} to higher dimensional $SO(n)$: 

\begin{theorem}\label{thm:geometric_example_SO5_intro}
The class $\frac{1}{2}[g]$ 
is a non-trivial element in the kernel of 
\begin{align*}
\check{\tau}  \colon \check{MU}(4)^4(SO(5)) \longrightarrow \check{H}(4)^4(SO(5)).  
\end{align*}
\end{theorem}

As in remark \ref{rem:image_of_ktheory_intro}, we could have formulated theorem \ref{thm:kernel_diff_cob_intro} for $ku$ instead of $MU$ as well, and the corresponding assumption would apply to the groups where surjectivity fails for $ku$ already. 
The geometric construction of theorem \ref{thm:geometric_example_SO5_intro}, however, and its generalisation to higher $SO(n)$ are particular to $MU$. 
Moreover, since the Thom morphism allows for a unified picture, we formulate our findings for $\tau$. 

Finally, we note that the phenomenon the example of theorem \ref{thm:geometric_example_SO5_intro} detects bears a certain similarity with the example used in \cite[\S 2.7]{hs} to explain the behavior of a certain partition function in mathematical physics. 
We refer for example to \cite[Example 48]{gradysatiAHSS} for other interesting phenomena in mathematical physics related to the study of the morphisms between generalised differential cohomology theories. 
We do not know of a potential similar application of theorem \ref{thm:geometric_example_SO5_intro} yet. 
We hope that the techniques to prove theorem \ref{thm:geometric_example_SO5_intro} will be useful to shed new light on the Abel--Jacobi invariant for complex cobordism of \cite{hausquick} and \cite{hfc}. \\


{\bf Acknowledgements:}
The first-named author would like to thank Tobias Barthel for helpful discussions and the Max Planck Institute for Mathematics in Bonn for its hospitality where parts of the work on this paper has been carried out. 
The second-named author was partially supported by the RCN Project No.\,313472 {\it Equations in Motivic Homotopy}. 
Both authors thank Knut Bjarte Haus for helpful discussions on the ideas used in section \ref{sec:geometric_examples}.

\section{Obstructions and detecting elements in the kernel}\label{sec:obstructions_and_detection}

In this section we explain the techniques that we use in section \ref{sec:Lie_groups} to study the cokernel of $\tau$. 
We assume that $X$ is a finite CW-complex for simplicity. 

\subsection{The Thom morphism is an edge map}\label{sec:surjective_Thom}

A key tool in our study of the Thom homomorphism is the Atiyah--Hirzebruch spectral sequence
\begin{align*}
    E_2^{p,q} = H^p(X;MU^q) \implies MU^{p+q}(X).
\end{align*}
Since $MU^\ast \cong \Z[x_{-2},x_{-4},\ldots]$, this spectral sequence is concentrated in the fourth quadrant. 
Since the top row of the $E_2$-page is the integral cohomology of $X$, there is a well-defined edge map such that the composition
\begin{align*}
    MU^p(X) \longrightarrow E_\infty^{p,0} \longrightarrow E_2^{p,0}\cong H^p(X;\Z)
\end{align*}
can be identified with the Thom morphism.  

It then follows from the general theory of spectral sequences that the Thom morphism is surjective if and only if all the differentials starting in the top row of the spectral sequence are trivial. 
Since all the differentials are torsion, the Thom morphism is surjective whenever $H^\ast(X;\Z)$ has no torsion.  

If $H^\ast(X;\Z)$ has torsion, the Thom morphism may still be surjective. 
Since the construction of the spectral sequence is functorial, the differentials starting in the top row of the $E_2$-page are cohomology operations of the form $d \colon H^*(X;\Z) \rightarrow H^{*}(X;A)$ where $A$ is a finitely generated free abelian group.  
If a differential $d$ is 
$p$-torsion, then so is 
the composition $\rho \circ d$, where $\rho$ is the map induced by the reduction modulo $p$ homomorphism of $A$.  
Thus we can describe all differentials using  cohomology operations of type $(\Z,m;\Z/p,n)$. 
These operations correspond to the elements in the cohomology group $H^n(K(\Z,m);\Z/p)$.  

For $p=2$, the cohomology ring $H^\ast(K(\Z,m);\Z/2)$ is a polynomial ring over generators of the form $\Sq^I(\iota_m)$, where $I$ is an admissible sequence where the last term is different from $1$, and $\iota_m$ is the fundamental class of $K(\Z,m)$ as explained in \cite[Chapter 9, Theorem 3]{mt2}. 
Thus, in order to prove that there are no non-trivial differentials that are $2$-torsion, it suffices to check that all Steenrod operations of odd degree are trivial (except $\Sq^1$, since a non-trivial differential increases the cohomological degree by at least $3$).

For odd primes $p$, the cohomology operations we have to study can all be described using the reduced power operations $P^k$ combined with Bocksteins $\beta$ (see \cite{cartan} for a complete description). 
In order to prove surjectivity it therefore suffices to show that all sequences of reduced power operations and Bocksteins that increase the cohomological degree by an odd number greater than $1$ must be trivial. 
The fact that this also works in cases where we have torsion of the form $\Z/{p^k}$ with $k>1$ can be deduced by considering short exact sequences of the form $\Z/p \rightarrow \Z/{p^k} \rightarrow \Z/{p^{k-1}}$.

\subsection{Obstructions and Bockstein cohomology}\label{sec:obstructions}

Now we explain how we can find cohomology classes which are not in the image of the Thom homomorphism. 
From the description of $\tau$ as an edge map we know that 
an element $x\in H^n(X;\Z)$ is not in the image of $\tau$ if there is at least one differential $d$ on the $E_2$-page of the Atiyah--Hirzebruch spectral sequence with $(\rho \circ d)(x)\ne 0$ where   
\begin{align*}
\rho \colon H^\ast(X;\Z) \longrightarrow H^\ast(X;\Z/p)
\end{align*}
is the homomorphism induced by reduction mod $p$. 
Suppose now we know how the Steenrod algebra acts on $H^\ast(X;\Z/p)$. 
In fact, all Steenrod operations of odd degree vanish on the image of $MU^*(X)$ in $H^*(X;\Z/p)$ for all prime numbers $p$ (see for example \cite[page 468]{totaro}, \cite[Proposition 3.6]{BenoistOttem},  \cite{cartanSem}). 
Then it remains to understand how $\rho$ acts.  
The tool we use to find the concrete element in $H^\ast(X;\Z/p)$ a given $x \in H^\ast(X;\Z)$ maps to is \emph{Bockstein cohomology}, the definition of which we now recall from \cite[Chapter 3E]{hatcher}:

The Bockstein homomorphism $\beta \colon H^n(X;\Z/p) \longrightarrow H^{n+1}(X;\Z/p)$ is the connecting homomorphism in the long exact sequence induced in cohomology by the short exact sequence $0 \longrightarrow \Z/p \longrightarrow \Z/{p^2} \longrightarrow \Z/p \longrightarrow 0$. 
It satisfies $\beta^2=0$, and thus defines a chain complex
\begin{equation*}
\begin{tikzcd}
\cdots \arrow[r, "\beta_{n-1}"] & H^n(X;\Z/p) \arrow[r, "\beta_n"] & H^{n+1}(X;\Z/p) \arrow[r, "\beta_{n+1}"] & \cdots.
\end{tikzcd}
\end{equation*}
The \emph{$n$th Bockstein cohomology} of $X$ is defined to be 
\begin{align*}
BH^n(X;\Z/p) := \frac{\Ker \beta_n}{\Imm \beta_{n-1}}.
\end{align*}
We compute the groups $BH^n(X;\Z/p)$ by providing concrete descriptions of the Bockstein complex. 
Since $X$ is assumed to be a finite CW-complex, all cohomology groups of $X$ are finitely generated.  
By \cite[Proposition 3E.3]{hatcher} 
the relationship between $H^\ast(X;\Z)$ and $BH^\ast(X;\Z/p)$ is then given as follows :
\begin{itemize}
    \item Each $\Z$-summand of $H^n(X;\Z)$ contributes one $\Z/p$-summand to $BH^n(X;\Z/p)$.
    \item Each $\Z/p$-summand of $H^n(X;\Z)$ contributes nothing to $BH^n(X;\Z/p)$.
    \item Each $\Z/{p^k}$-summand (with $k\geq 2$) of $H^n(X;\Z)$ contributes one $\Z/p$-summand to $BH^{n-1}(X;\Z/p)$ and one $\Z/p$-summand to $BH^n(X;\Z/p)$.
\end{itemize}

Finally, for an odd prime $p$, we will also use the following  obstruction. 

\begin{lemma}\label{lem:Milnor_obstruction}
Let $Q_1 \colon H^*(X;\Z/p) \to H^{*+2p-1}(X;\Z/p)$ be the first Milnor operation and let $x \in H^i(X;\Z)$ be a non-torsion class. 
If $Q_1(\rho(x)) \ne 0$, then $x$ is not in the image of $ku^i(X) \to H^i(X;\Z)$ and hence not in the image of the Thom morphism. 
\end{lemma}
\begin{proof}
By \cite[Proposition 1.7]{wilson} (see also  \cite[Proposition 4-4]{tamanoi00}), there is a commutative diagram 
\begin{align*}
\xymatrix{
ku_{(p)}^i(X) \ar[r]^-{\tau_{ku_{(p)}}} & H^i(X;\Z_{(p)}) \ar[d] \ar[r]^-{\cdot v_1} & ku_{(p)}^{i+2p-1}(X) \ar[d] \\
 & H^i(X;\Z/p) \ar[r]_-{\pm Q_1} & H^{i+2p-1}(X;\Z/p)}
\end{align*}
in which the top row is exact, where $ku_{(p)}^i(X)$ denotes $p$-local connective complex $K$-theory and the map $\tau_{ku_{(p)}}$ is the map which factors the canonical morphism $\tau_{BP} \colon BP \to H\Z_{(p)}$ for Brown--Peterson theory. 
Thus, if $Q_1(\rho(x)) \ne 0$, then the image of $x$ in $H^i(X;\Z_{(p)})$ cannot be lifted to $ku_{(p)}^i(X)$. 
This implies that $x$ cannot be lifted to $ku^i(X)$ either,  
and hence the assertion. 
\end{proof}


\subsection{The kernel of the differential Thom morphism}\label{sec:differential_kernel}

We will now explain  how the failure of $\tau$ to be surjective enables us to find non-trivial elements in the kernel of $\check{\tau}$.  
We write $MU^{*<0}\cdot MU^k(X)$ for the subgroup of $MU^k(X)$ consisting of elements of the form $\gamma \cdot \mu$ where $\gamma \in MU^{k-s}$ and $\mu \in MU^{s}(X)$ with $s>k$.  
The sum over all $k$ defines an ideal in $MU^*(X)$ which we denote by $MU^{*<0}\cdot MU^*(X)$.  
Since $\tau(MU^{*<0}) = 0$, we get that $\tau$ induces a well-defined homomorphism 
\begin{equation*}
\tau \colon MU^*(X)/(MU^{*<0}\cdot MU^*(X)) \to H^*(X;\Z) ,
\end{equation*}
which we also denote by $\tau$. 
Consider the homomorphism $MU^* \to \Z$ which sends $n \cdot 1 \in MU^0$ to $n\in \Z$ and $\gamma \in MU^{*<0}$ to $0$. 
Then there is an isomorphism of rings
\begin{equation*}
MU^*(X)/(MU^{*<0}\cdot MU^*(X)) \cong MU^*(X)\otimes_{MU^*}\Z. 
\end{equation*}
By slight abuse of notation, we then also write $MU^k(X)\otimes_{MU^\ast} \Z$ for the group $MU^k(X)/(\oplus_s MU^{s}\cdot MU^{k-s}(X))$. 
Now we let $MU^*$ act on $\rz$ by the map ${MU^\ast \otimes \rz \longrightarrow \rz}$ defined by 
\begin{align*}
 n \otimes a & \longmapsto na, ~ \text{for} ~ n\in MU^0 \cong \Z \\
 \gamma \otimes a & \longmapsto 0, ~ \text{for} ~ \gamma \in MU^{*<0}.
\end{align*}
Then we get a canonical isomorphism  
\begin{align*} 
\left(MU^*(X)\otimes_{MU^\ast} \Z \right) \otimes_\Z \rz \xrightarrow{\cong} MU^*(X) \otimes_{MU^*} \rz. 
\end{align*}
We will now explain how the information on the cokernel of $\tau$ helps to understand the kernel of the induced Thom homomorphism 
\begin{align*}
\btrz \colon MU^*(X) \otimes_{MU^\ast} \rz \longrightarrow H^*(X;\Z)\otimes_{\Z} \rz
\end{align*}
in differential cobordism.   
\begin{lemma}\label{lem:detecting_nontrivial_elements_in _kernel}
Let $\alpha \in H^k(X;\mathbb{Z})$ be a non-torsion class. 
Assume that the image of the Thom morphism 
\[
\tau \colon MU^k(X) \longrightarrow H^k(X;\mathbb{Z})
\]
contains $n\alpha$ for some integer $n> 1$, but not $\alpha$ itself or an element of the form $\alpha+y$, where $n\cdot y=0$. 
Let $\mu \in MU^k(X)$ be an element such that $\tau(\mu)=n\alpha$.  
Then 
\begin{align*}
    \mu \otimes \frac{1}{n} \in MU^k(X) \otimes_{MU^*} \rz 
\end{align*} 
is a non-trivial element in the kernel of the induced  map
\begin{align*}
\btrz \colon MU^k(X) \otimes_{MU^*} \rz \longrightarrow H^*(X;\Z)\otimes_{\Z} \rz.   
\end{align*}
\end{lemma}
\begin{proof}
The element 
$\mu \otimes \frac{1}{n}$ 
maps to $0$ under $\btrz$ since $n\alpha \otimes \frac{1}{n} = \alpha \otimes 1 = 0$. 
However, $\mu \otimes \frac{1}{n}$ cannot be $0$ in $MU^k(X) \otimes_{MU^\ast} \rz$, since if $\mu$ had been of the form $n \gamma$, then $\gamma$ would map to $\alpha$ or $\alpha+y$. 
\end{proof}


\subsection{Detecting elements in the kernel of the differential Thom morphism}\label{sec:detecting_kernel}

We now describe a procedure to find an element $\mu$ as in lemma \ref{lem:detecting_nontrivial_elements_in _kernel} using the Atiyah--Hirzebruch spectral sequence. 
In section \ref{sec:geometric_examples} we will give a geometric construction for special orthogonal groups. 
For the other cases, we can proceed as follows. 
Assume that we have a non-torsion cohomology class $\alpha \in H^k(X;\Z)$ which is not in the image of the Thom morphism, while an integer multiple $k\alpha$ is in the image. 
We will now explain how we can then find a cobordism class which maps to $n\alpha$. 
Since $\alpha$ is not in the image of the Thom morphism, there must be at least one non-trivial differential starting at $H^k(X;\Z)$. 
If this differential is, say, $m$-torsion, then $m \alpha$ is in the kernel of the differential and survives to the next page of the spectral sequence. 
Since $X$ is assumed to be finite dimensional, the spectral sequence is bounded on the right, and there can only be finitely many non-trivial differentials starting at any one position. 
By counting how much torsion there is in cohomological degrees greater than $n$, we can then determine an integer $n$ for which $n \alpha$ must be in the image of the Thom morphism. 
Once we have reached the $E_\infty$-page of the spectral sequence, the position $(k,0)$ contains the desired cobordism class.


\section{The cokernel for compact Lie groups}\label{sec:Lie_groups}

The goal of this section is to determine whether or not the Thom morphism is surjective for a given compact, connected, simple Lie group. 
Such a Lie group has a simple Lie algebra. 
Given a simple Lie algebra $\mathfrak{g}$, we find the associated Lie groups using the following method based on \cite[10.7.2, Theorem 4]{procesi}. 
We first determine the unique (up to isomorphism) compact, simply-connected Lie group $G$ with Lie algebra $\mathfrak{g}$. The center $Z(G)$ is always finite. 
The other compact, connected, simple Lie groups with the same Lie algebra are of the form $G/K$, where $K$ is a subgroup of $Z(G)$.
Organising our analysis by the associated Lie algebra is justified by the following observation. 
Given a Lie group $G$, we denote by $\Hfree{\ast}{G}$ the non-torsion part of the cohomology $H^\ast(G;\Z)$. 
Then there is an isomorphism $\Hfree{\ast}{G} \cong \Hfree{\ast}{H}$ if $G$ and $H$ are Lie groups with the same Lie algebra.  
We will therefore recall the non-torsion cohomology part only once in the section for a given Lie algebra.

Unless otherwise stated, the computation of the  cohomology rings can be found in one of the following two sources: 
The cohomology of the groups $SU(n)$, $Sp(n)$, $\Spin(n)$, $SO(n)$ as well as all the exceptional Lie groups and classifying spaces can be found in  \cite{mt}, while the cohomology of $Ss(n)$, $PSO(n)$, $PSp(n)$ and the quotients of $SU(n)$ can be found in  \cite{bb}. 
Finally, given a ring $R$ we write $\Lambda_R(x_{i_1},\ldots,x_{i_n}):= R[x_{i_1},\ldots,x_{i_n}]/(x_{i_1}^2,\ldots,x_{i_n}^2)$, and unless otherwise stated $x_{i_j}$ is an element of degree $i_j$.  
When the choice of ring is clear from the context, we omit $R$ from the notation.



\subsection{Groups with Lie algebra \texorpdfstring{$\mathfrak{b_n}$}{Bn} and \texorpdfstring{$\mathfrak{d_n}$}{Dn}}

The simply-connected Lie groups that correspond to the Lie algebras of type $\mathfrak{b_n}$ and $\mathfrak{d_n}$ are the \textit{spin groups} $\Spin(2n+1)$ and $\Spin(2n)$, respectively. 
We will consider both types of spin groups together, since their cohomology rings are similar. 
However, the possible quotients are different in the odd and even cases. 
The center of $\Spin(2n+1)$ is isomorphic to $\Z/2$, which gives us only one possible quotient, the \textit{odd special orthogonal group}, denoted $SO(2n+1)$. 

For the even case, we know by \cite[Chapter II, Theorem 4.14]{mt} that the centers are given by $Z(\Spin(4n+2)) \cong \Z/4$ and $Z(\Spin(4n)) \cong \Z/2 \oplus \Z/2$. 
For $\Spin(4n+2)$, taking the quotient by the subgroup of order $2$ yields the \textit{even special orthogonal group} $SO(4n+2)$, while taking the quotient by the whole center gives the \textit{projective special orthogonal group} $PSO(4n+2)$. 
For $\Spin(4n)$, the center $\Z/2 \oplus \Z/2$ has three subgroups of order $2$. 
Taking the quotient by the whole center once again produces a projective special orthogonal group $PSO(4n)$. 
One of the subgroups of order $2$ will again give us a special orthogonal group $SO(4n)$. 
The remaining two subgroups produce isomorphic quotient groups, known as the \textit{semi-spin group} $Ss(4n)$ (see \cite[Chapter II, Theorem 4.15]{mt}). 
In total we have to consider four different types of groups. 

\subsubsection{Special orthogonal groups}\label{subsec:Thom_for_SOn}

The non-torsion cohomology of the special orthogonal groups is given by
\begin{align*}
    \Hfree{\ast}{SO(n)} & \cong \begin{cases}
    \Lambda(e_3,e_7,\ldots,e_{2n-3}), & n \text{ odd} \\
    \Lambda(e_3,e_7,\ldots,e_{2n-5}, y_{n-1}), & n \text{ even}.
    \end{cases}
\end{align*}

\begin{prop}\label{prop:SOn_for_n_ge_5}
For $n\ge 5$, the generator $e_3 \in H^3(SO(n);\Z)$ is not in the image of the Thom homomorphism.
\end{prop}
\begin{proof}
The $\Z/2$-cohomology of $SO(n)$ is given by 
\begin{align} \label{SOmod2}
    H^\ast(SO(n);\Z/2) \cong \Z/2[u_1,u_3,\ldots, u_{2m-1}]/(u_1^{k_1},u_3^{k_3},\ldots, u_{2m-1}^{k_{2m-1}}),
\end{align}
where $m=\lfloor \frac{n}{2}\rfloor$ and $k_i$ is the least power of $2$ such that $\vert u_i^{k_i} \vert \geq n$.  
In order to find the image of $e_3$ in $H^3(SO(n);\Z/2)$ under $\rho$, we need to analyse the Bockstein homomorphism $\beta$.  
From \cite{hatcher}, we have 
\begin{align*}
    \beta(u_{2i-1}) = u_{2i} \text{ and } \beta(u_{2i}) = 0,
\end{align*}
where we interpret $u_{2i}$ as $u_i^2$ and iterate if necessary. 
Assuming $n \geq 5$, we have the following table for the generators for $H^\ast(SO(n);\Z/2)$ in low degrees, where the arrows denote the non-trivial Bockstein homomorphisms.

\begin{equation} \label{sobockstein}
\begin{tikzcd}
\text{Degree:} & 1 & 2 & 3 & 4 \\
\text{Generators:} & u_1 \arrow[r] & u_1^2 & u_1^3 \arrow[r] & u_1^4 \\
& & & u_3 \arrow[ur] & u_1u_3
\end{tikzcd}
\end{equation}
We see that $BH^3(SO(n);\Z/2)$ is generated by $u_1^3 + u_3$. It follows that the reduction map to $\Z/2$-cohomology maps $e_3$ to $u_1^3 + u_3$. 
We can then deduce that $e_3$ is not in the image of the Thom homomorphism, since
\begin{align*}
\Sq^3(u_1^3 + u_3) = u_1^6 + u_3^2 \neq 0.
\end{align*}
Note that $u_3^2 = 0$ if $n=5$, but $u_1^6$ is nonzero.
\end{proof}

\begin{remark}
Using the same methods, we can show that any given generator $e_{4k+3} \in H^{4k+3}(SO(n);\Z)$ is not in the image of the Thom morphism for sufficiently large $n$. 
However, we do not know of an efficient way to determine a minimal $n$ for each generator $e_{4k+3}$, apart from analysing the Bockstein diagrams on a case by case basis. 
We return to this question in section \ref{sec:generalize_to_higher_SOn}. 
\end{remark}

\begin{prop}\label{prop:SOn_for_n_le_5}
For $SO(n)$ with $n \leq 4$, the Thom morphism is surjective in all degrees.
\end{prop}
\begin{proof}
We have the homeomorphisms
\begin{align*}
SO(1) \cong \mathrm{pt}, ~ 
SO(2) \cong S^1, ~ 
SO(3) \cong \RP^3, ~ \text{and} ~ 
SO(4) \cong \RP^3 \times S^3.
\end{align*}
For $SO(1)$ and $SO(2)$, the surjectivity follows from the fact that the Thom morphism is surjective in degrees $\leq 2$. 
For $SO(3)$, we use the same fact for degrees $\leq 2$. 
Moreover, there is no nontrivial differential in the Atiyah--Hirzebruch spectral sequence starting in cohomological degree $3$, since $\RP^3$ is $3$-dimensional. 
This shows that the Thom morphism is also surjective in all degrees for $SO(3)$. 
Finally, the integral cohomology of $SO(4)$ is given by
\begin{align*}
    H^k(SO(4);\Z) \cong \begin{cases}
    \Z, \qquad &k=0,6 \\
    \Z \oplus \Z \quad &k=3 \\
    \Z/2, \quad &k=2,5 \\
    0, \quad &\text{ else.}
    \end{cases}
\end{align*}
Again, there are no differentials which start in degree $\geq 3$, increase the cohomological degree by at least $3$, and which end in torsion. 
Thus, the Thom morphism is surjective for $SO(4)$.  
\end{proof}

\subsubsection{Spin groups}

For the rest of this section, we use the following notation. Given $n\in \mathbb{N}$, we let $q$ be the greatest power of $2$ such that $q|n$, and let $t$ be the least power of $2$ such that $n \leq t$.  
The cohomology of the spin groups is given by
\begin{align*}
    H^\ast(\Spin(n);\Z/2) \cong \Lambda(z) \otimes \Z/2[u_3,u_5,\ldots,u_{2m-1}] / (u_3^{k_3},\ldots,u_{2m-1}^{k_{2m-1}}),
\end{align*}
where $\vert z\vert=t-1$ and where $m$ and the $k_i$'s are as in \eqref{SOmod2}.
\begin{prop}
    For $n \geq 7$, the generator $e_3 \in H^3(\Spin(n);\Z)$ is not in the image of the Thom morphism. For $n \leq 6$, the Thom morphism is surjective in all degrees.
\end{prop}
\begin{proof}
    For $n \leq 7$, there is no torsion in the integral cohomology of $\Spin(n)$, and it follows that the Thom morphism is surjective. However, for $n\geq7$, the generator $e_3 \in H^3(\Spin(n);\Z)$ maps to $u_3\in H^3(\Spin(n);\Z/2)$, for which 
\begin{align*}
\Sq^3 u_3 = u_3^2 \neq 0.
\end{align*}
Thus, $e_3$ is not in the image of the Thom morphism for $n \geq 7$.
\end{proof}

\subsubsection{Semi-spin groups}

The cohomology of the semi-spin groups with coefficients in $\Z/2$ is given by 
\begin{align*}
    &\; H^\ast(Ss(n);\Z/2) \\ \cong &\; \Z/2[v]/(v^q) \otimes \Lambda(z) \otimes \Z/2[u_3,u_5,\ldots, \widehat{u}_{q-1},\ldots,u_{n-1},u_{2q-2}]/(u_3^{k_3},\ldots,u_{n-1}^{k_{n-1}},u_{2q-2}^{k_{2q-2}}),
\end{align*}
where $\vert v \vert = 1$ and $\vert z \vert = t-1$. 
The Steenrod operations are given by
\begin{align*}
    \Sq^j(u_k) = \binom{k}{j} u_{k+j}
\end{align*}
wherever it makes sense, with the exception that
\begin{align} \label{steenrodexcpetionss}
    \Sq^1(u_k) = v^{k+1}
\end{align}
if $q \geq 8$ and $k = \frac{q}{2}-1$. 
Recall that since we are dealing with semi-spin groups, $n$ must be a multiple of $4$. 
Hence we do not need separate cases for even and odd $n$. 
Note also that the class $u_{2q-2}$ will only be included if $2q-2 < n$. 

\begin{prop}
For $Ss(4)$, the Thom morphism is surjective in all degrees. 
For $k\ge 2$, we have: 
If $8 \mid n$, then the generator $e_3 \in H^3(Ss(4k);\Z)$ is not in the image of $\tau$. 
If $8 \nmid n$, then the generator $e_7 \in H^7(Ss(4k);\Z)$ is not in the image of $\tau$. 
\end{prop}
\begin{proof}
For $n=4$, the cohomology ring together with its  Steenrod operations of $Ss(4)$ is isomorphic to the cohomology ring with Steenrod operations of $SO(4)$. 
It then follows from proposition \ref{prop:SOn_for_n_le_5} that the Thom morphism is surjective for $Ss(4)$.

Now we assume $n=4k$ and $k\ge 2$. 
There are three cases to consider:\\ 

\textbf{Case 1:} $\fmodu{n}{8}{16}$ 

In this case, $q = 8$, and the $\Z/2$-cohomology ring is given by  
\begin{align*}
    H^\ast(Ss(n);\Z/2) \cong \Z/2[v]/(v^8) \otimes \Lambda(z) \otimes \Z/2[u_3,\ldots,\widehat{u}_7,\ldots,u_{n-1},u_{14}]/(u_3^{k_3},\ldots),
\end{align*}
where $\vert z \vert \geq 7$ since $t$ is at least $8$. Since $q=8$, we get that
\begin{align*}
    \Sq^1(u_3) = v^4
\end{align*}
by equation \eqref{steenrodexcpetionss}. 
The other $\Sq^1$s are easy to work out, which leads to the following Bockstein diagram in low degrees:
\begin{equation*}
\begin{tikzcd}
\text{Degree:} & 1 & 2 & 3 & 4 \\
\text{Generators:} & v \arrow[r] & v^2 & v^3 \arrow[r] & v^4 \\
& & & u_3 \arrow[ur] & v u_3
\end{tikzcd}
\end{equation*}
The similarity to diagram \eqref{sobockstein} is not coincidental, since there is an isomorphism $Ss(8) \cong SO(8)$, and the other semi-spin groups of this form look similar in low degrees. 
The non-torsion class $e_3 \in H^3(Ss(n);\Z)$ maps to $v^3 + u_3 \in H^3(Ss(n);\Z/2)$, and we can check that $\Sq^3$ does not act trivially on this class:
\begin{align*}
    \Sq^3(v^3 + u_3) = v^6 + u_3^2 \neq 0.
\end{align*}

\textbf{Case 2:} $\fmodu{n}{0}{16}$ 

In this case, since $q$ is now greater than $8$, the Bockstein homomorphism acts trivially on $u_3$.
\begin{equation*}
\begin{tikzcd}
\text{Degree:} & 1 & 2 & 3 & 4 \\
\text{Generators:} & v \arrow[r] & v^2 & v^3 \arrow[r] & v^4 \\
& & & u_3 & v u_3
\end{tikzcd}
\end{equation*}
Moreover, the Bockstein cohomology in degree $3$ is  generated by $u_3$ alone. 
This implies that $e_3$ is sent to $u_3$, and we see that $u_3$ is not in the image of the Thom morphism. \\

\textbf{Case 3:} $\fmodu{n}{4}{8}$

We have $q=4$ and $n \geq 12$, which means that the $\Z/2$-cohomology is given by
\begin{align*}
    H^\ast(Ss(n);\Z/2) \cong \Z/2[v]/(v^4) \otimes \Lambda(z) \otimes \Z/2[u_5,u_6,u_7,u_9,\ldots,u_{n-1}]/(u_5^{k_5},\ldots),
\end{align*}
where we see that $\vert z \vert \geq 15$. 
We get the Bockstein diagram
\begin{equation*}
\begin{tikzcd}[row sep=small]
1 & 2 & 3 & 4 & 5 & 6 & 7 & 8 \\
v \arrow[r] & v^2 & v^3  \\
& & & & u_5 \arrow[rd] & v u_5 \arrow[rd] \arrow[r] & v^2 u_5 \arrow[rd] & v^3 u_5 \\
& & & & & u_6 & v u_6 \arrow[r] & v^2 u_6 \\
& & & & & & u_7 & v u_7
\end{tikzcd}
\end{equation*}
where the two arrows starting in $v u_5$ indicate that the Bockstein is given by a sum, i.e.,  $\beta(vu_5) = v^2 u_5 + v u_6$.
In this case, $e_3$ is in the image of the Thom morphism, since its reduction $v^3$ does not survive any Steenrod operations. 
However, we can find a suitable class in degree $7$. 
From the Bockstein diagram we see that $e_7$ is sent to either $u_7$ or $u_7 + v^2u_5 + vu_6$, and we can check that both classes survive $\Sq^3$:
\begin{align*}
    \Sq^3(u_7) = \binom{7}{3}u_{10} &= u_5^2 \neq 0 \\
    \Sq^3(u_7 + v^2 u_5 + v u_6) &= u_5^2 \neq 0. \qedhere 
\end{align*}
\end{proof}


\subsubsection{Projective special orthogonal groups} 

The $\Z/2$-cohomology of the projective special orthogonal groups is given by 
\begin{align*}
    & \; H^\ast(PSO(n);\Z/2)  \\
\cong &  \; \Z/2[v]/(v^q) \otimes \Z/2[u_1,u_3,\ldots,\widehat{u}_{q-1},\ldots,u_{n-1},u_{2q-2}] / (u_1^{k_1},\ldots,u_{n-1}^{k_{n-1}},u_{2q-2}^{k_{2q-2}}),
\end{align*}
where $\vert v \vert=1$, with the Steenrod operations acting by 
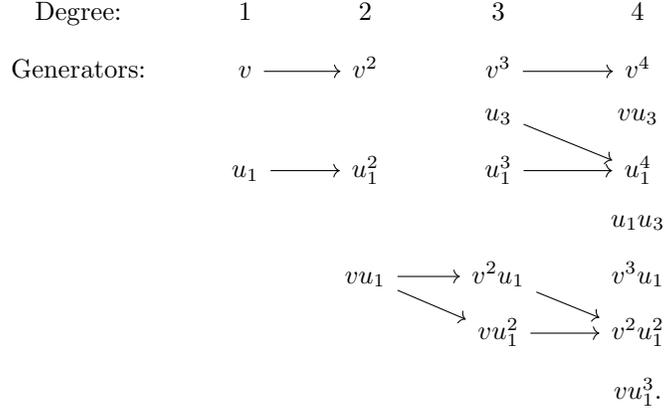
\begin{align*}
    \Sq^j u_k = \binom{k}{j} u_{k+j},
\end{align*}
whenever it makes sense, except when $j=1,\; k=\frac{q}{2}-1$ and $q\geq 8$, in which case
\begin{align*}
    \Sq^1(u_{k}) = u_{k+1} + v^{k+1}.
\end{align*}
Note that $u_{2q-2}$ is only be included if $2q-2<n$, as for the semi-spin groups.
If $n$ is odd, then $PSO(n)=SO(n)$. 
Hence we will focus on the case that $n$ is even. 

\begin{prop}
Let $n \ge 8$ be even. If $8 \mid n$, then the generator $e_3\in H^3(PSO(n);\Z)$ is not in the image of $\tau$. 
If $8 \nmid n$, then the generator $e_7\in H^7(PSO(n);\Z)$ is not in the image of $\tau$. 
\end{prop}
\begin{proof}
We have to consider the following cases: \\ 

\textbf{Case 1:} $\fmodu{n}{8}{16}$

We have  $q=8$ and note that $\Sq^1(u_3) = u_1^4 + v^4$.  
We get the following diagram of Bockstein homomorphisms:
\begin{equation*}
\begin{tikzcd}[row sep=small]
\text{Degree:} & 1 & 2 & 3 & 4 \\
\text{Generators:} & v \arrow[r] & v^2 & v^3 \arrow[r] & |[alias=vvvv]| v^4 \\
& & & u_3 \arrow[to=vvvv] \arrow[to=14] & vu_3 \\
 & u_1 \arrow[r] & u_1^2 & u_1^3 \arrow[r] & |[alias=14]| u_1^4 \\
& & & & u_1 u_3 \\
& & v u_1 \arrow[r] \arrow[dr] & v^2 u_1 \arrow[dr] & v^3 u_1 \\
& & & v u_1^2 \arrow[r] & v^2 u_1^2 \\
& & & & v u_1^3
\end{tikzcd}
\end{equation*}
The non-torsion class $e_3 \in H^3(PSO(n);\Z)$ maps to either $v^3+u_1^3+u_3$ or $v^3+u_1^3+u_3 + v^2u_1 + vu_1^2$ in $H^3(PSO(n);\Z/2)$, and we have 
\begin{align*}
    &\Sq^3(v^3+u_1^3+u_3) = v^6 + u_1^6 + u_3^2 \neq 0 \\
    &\Sq^3(v^3+u_1^3+u_3 v^2u_1 + vu_1^2) = v^6 + u_1^6 + u_3^2 + v^4u_1^2 + v^2u_1^4 \neq 0.
\end{align*}

\textbf{Case 2:} $\fmodu{n}{0}{16}$

In low degrees, this is almost the same as the previous case, with the exception that $\Sq^1(u_3)=u_1^4$ since $q \geq 16$. This gives us the diagram

\begin{equation*}
\begin{tikzcd}[row sep=tiny]
\text{Degree:} & 1 & 2 & 3 & 4 \\
\text{Generators:} & v \arrow[r] & v^2 & v^3 \arrow[r] & |[alias=vvvv]| v^4 \\
& & & u_3 \arrow[to=14] & vu_3 \\
 & u_1 \arrow[r] & u_1^2 & u_1^3 \arrow[r] & |[alias=14]| u_1^4 \\
& & & & u_1 u_3 \\
& & v u_1 \arrow[r] \arrow[dr] & v^2 u_1 \arrow[dr] & v^3 u_1 \\
& & & v u_1^2 \arrow[r] & v^2 u_1^2 \\
& & & & v u_1^3.
\end{tikzcd}
\end{equation*}
Then $e_3\in H^3(PSO(n);\Z)$ maps to $u_1^3 + u_3$ or $u_1^3 + u_3 + v^2u_1 + v u_1^2$, for which
\begin{align*}
    &\Sq^3(u_1^3+u_3) = u_1^6 + u_3^2 \neq 0 \\
    &\Sq^3(u_1^3+u_3 + v^2u_1 + vu_1^2) = u_1^6 + u_3^2 + v^4u_1^2 + v^2u_1^4 \neq 0.
\end{align*}

\textbf{Case 3:} $\fmodu{n}{4}{8}$ 

Assume $n\geq 12$. 
In this case $q=4$, and consequently the class $u_3$ does not exist. 
The Bockstein diagram is

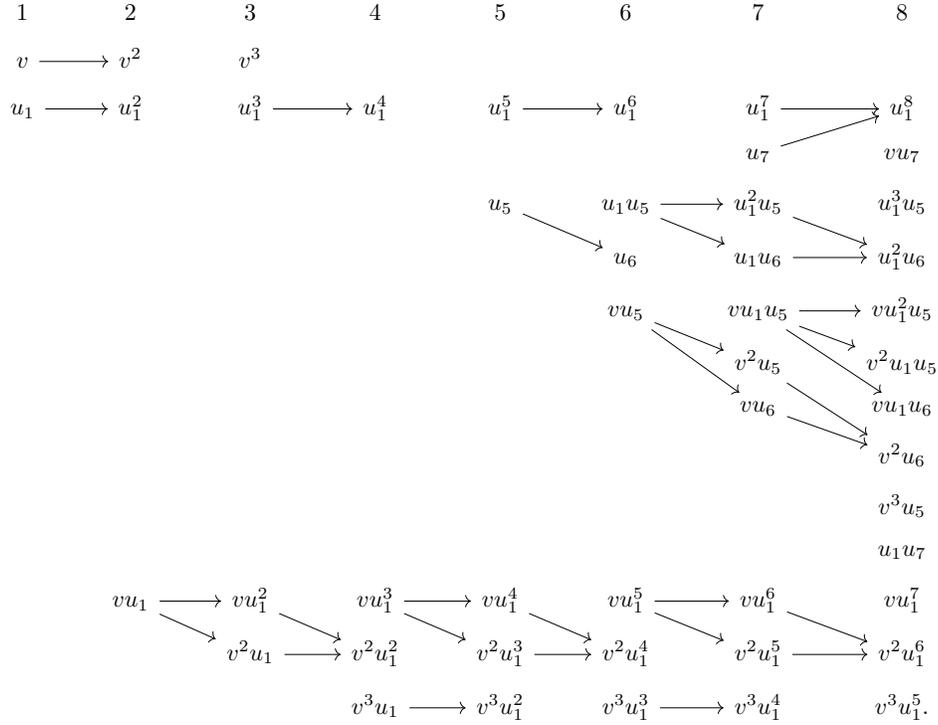
\begin{equation*}
\adjustbox{scale=0.9}{
\begin{tikzcd}[row sep=tiny]
1 & 2 & 3 & 4 & 5 & 6 & 7 & 8\\
v \arrow[r] & v^2 & v^3 \\
  u_1 \arrow[r] & u_1^2 & u_1^3 \arrow[r] & u_1^4& u_1^5 \arrow[r] & u_1^6 & u_1^7 \arrow[r] & |[alias=18]| u_1^8 \\
   & & & & & & |[alias=7]| u_7 \arrow[to=18] & |[alias=v7]| v u_7 \\
  & & & & |[alias=5]| u_5 \arrow[to=6] & |[alias=115]| u_1 u_5 \arrow[to=125] \arrow[to=116] & |[alias=125]| u_1^2 u_5 \arrow[to=126] & |[alias=135]| u_1^3 u_5 \\ 
   & & & & & |[alias=6]| u_6 & |[alias=116]| u_1 u_6 \arrow[to=126] & |[alias=126]| u_1^2 u_6 \\
  & & & & & |[alias=v5]| v u_5 \arrow[to=vv5] \arrow[to=v6] & |[alias=v115]| v u_1 u_5 \arrow[to=v125] \arrow[to=vv115] \arrow[to=v116] & |[alias=v125]| v u_1^2 u_5\\
  & & & & & & |[alias=vv5]| v^2 u_5 \arrow[to=vv6] &|[alias=vv115]| v^2 u_1 u_5\\
  & & & & & & |[alias=v6]| v u_6 \arrow[to=vv6] & |[alias=v116]| v u_1 u_6 \\
  & & & & & & & |[alias=vv6]| v^2 u_6 \\
   & & & & & & & |[alias=vvv5]| v^3 u_5 \\
  & & & & & & & u_1 u_7 \\
 & v u_1 \arrow[r] \arrow[dr] & v u_1^2 \arrow[dr] & v u_1^3 \arrow[r] \arrow[dr] & v u_1^4 \arrow[dr] & v u_1^5 \arrow[r] \arrow[dr] & v u_1^6 \arrow[dr] & v u_1^7 \\
  & & v^2 u_1 \arrow[r] & v^2 u_1^2 & v^2 u_1^3 \arrow[r] & v^2 u_1^4 & v^2 u_1^5 \arrow[r] & v^2 u_1^6  \\
  & & & v^3 u_1 \arrow[r] & v^3 u_1^2 & v^3 u_1^3 \arrow[r] & v^3 u_1^4 & v^3 u_1^5. \\
\end{tikzcd}
}
\end{equation*}
We see that the class $e_7$ can reduce to several different classes in $H^7(PSO(n);\Z/2)$ depending on our choice of isomorphism $H^7(PSO(n);\Z) \cong \Z \oplus \Z_2^4$. 
We have that $e_7$ maps to $u_7 + u_1^7$ or $u_7 + u_1^7$ plus any of the classes
$u_1^2 u_5 + u_1u_6$, $v^2u_5 + vu_6$, $vu_1^6 + v^2u_1^5$, $v^3u_1^5$. 
We have
 \begin{align*}
     \Sq^3(u_7 + u_1^7) = u_{10} + u_1^{10} \neq 0.
 \end{align*}
We then note that applying $\Sq^3$ to any of the torsion classes $u_1^2 u_5 + u_1u_6$, $v^2u_5 + vu_6$, $vu_1^6 + v^2u_1^5$, $v^3u_1^5$ cannot yield $u_{10}$. 
Hence they cannot cancel out the nonzero contribution we got from $\Sq^3(u_7 + u_1^7)$.  
This proves that $e_7$ and $e_7$ plus torsion are not in the image of the Thom morphism. \\

\textbf{Case 4:} $\fmodu{n}{2}{4}$

Assume $n \geq 10$. Since $q=2$, we there is no class $u_1$, but instead a class $u_2$. 
The cohomology and Bockstein diagrams are therefore
\begin{align*}
    H^\ast(PSO(n);\Z/2) \cong \Z/2[v]/(v^2) \otimes \Z/2[u_2,u_3,u_5,\ldots,u_{n-1}]/(u_2^{k_2},\ldots)
\end{align*}
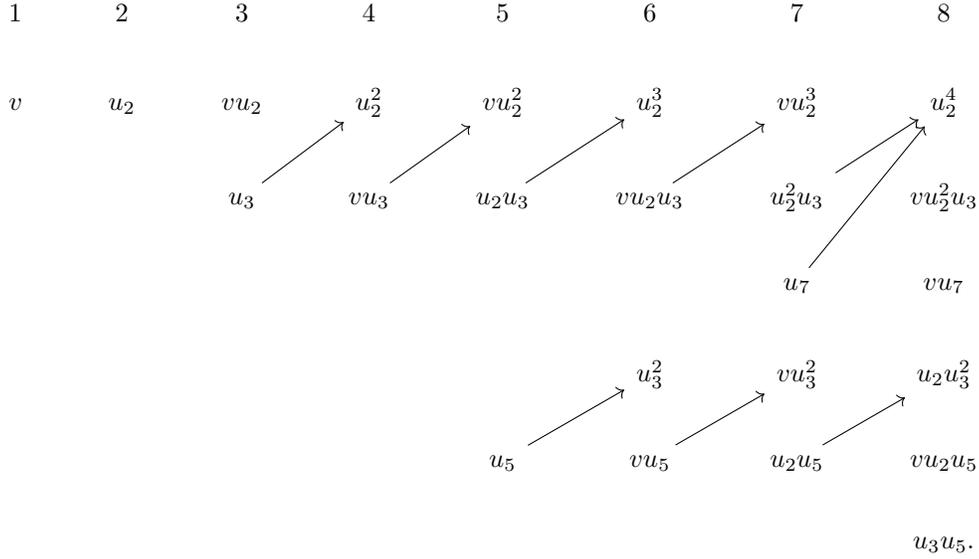
\begin{equation*}
\begin{tikzcd}
1 & 2 & 3 & 4 & 5 & 6 & 7 & 8\\
v & u_2 & v u_2 & u_2^2 & v u_2^2 & u_2^3 & v u_2^3 & |[alias=24]| u_2^4 \\
& & u_3 \arrow[ur] & v u_3 \arrow[ur] & u_2 u_3 \arrow[ur] & v u_2 u_3 \arrow[ur] & u_2^2 u_3 \arrow[to=24] & v u_2^2 u_3 \\
& & & & & & u_7 \arrow[to=24] & v u_7 \\
& & & & & u_3^2 & v u_3^2 & u_2 u_3^2 \\
& & & & u_5 \arrow[ur] & v u_5 \arrow[ur] & u_2 u_5 \arrow[ur] & v u_2 u_5 \\
& & & & & & & u_3 u_5.
\end{tikzcd}
\end{equation*}

The class $e_7$ is sent to either $u_7 + u_2^2 u_3$, or $u_7 + u_2^2 u_3$ plus one or both of the classes $v u_2^3$, $v u_3^2$. 
As above, we may use $\Sq^3$ as an obstruction, but the computation is easier if we use $\Sq^7$. 
We get 
\begin{align*}
    &\Sq^7(u_7 + u_2^2 u_3) = u_7^2 + u_2^4 u_3^2 \neq 0 \\
    &\Sq^7(v u_2^3) = v^2 u_2^6 = 0 \\
    &\Sq^7(v u_3^2) = v^2 u_3^4 = 0.
\end{align*}
Hence $e_7$ is not be in the image of the Thom morphism. 
Note that $u_7^2=0$ if $n \leq 14$, but $u_2^4 u_3^2$ is nonzero.
\end{proof}

\begin{prop}
For $PSO(2)$, $PSO(4)$ and $PSO(6)$, the Thom morphism is surjective in all degrees. 
\end{prop}
\begin{proof}
The isomorphism $PSO(2) \cong SO(2)$ implies that the Thom morphism is surjective for $PSO(2)$, by proposition \ref{prop:SOn_for_n_le_5}. 
For $PSO(4)$, the relevant cohomology rings are given by 
\begin{align*}
    &\Hfree{\ast}{PSO(4)} \cong \Lambda(e_3,y_3) \\
    &H^\ast(PSO(4);\Z/2) \cong \Z[v]/(v^4) \otimes \Z/2[u_1]/(u_1^4). 
\end{align*}
Moreover, we have the Bockstein diagram
\begin{equation*}
\begin{tikzcd}
\text{Degree:} & 1 & 2 & 3 & 4 & 5 & 6\\
\text{Generators:} & v \arrow[r] & v^2 & v^3 \\
 & u_1 \arrow[dr] & v u_1 \arrow[dr] \arrow[r] & v^2 u_1 \arrow[dr] & v^3 u_1 \arrow[dr] \\
 & & u_1^2 & v u_1^2 \arrow[r] & v^2 u_1^2 & v^3 u_1^2 \\
 & & & u_1^3 & v u_1^3 \arrow[r] & v^2 u_1^3 & v^3 u_1^3.
\end{tikzcd}    
\end{equation*}
It follows that the integral cohomology of $PSO(4)$ is given by 
\begin{align*}
    H^k(PSO(4);\Z) \cong \begin{cases}
    \Z, \qquad &k=0,6 \\
    \Z/2, \quad &k=4 \\
    \Z/2 \oplus \Z/2, \quad &k=2,5 \\
    \Z \oplus \Z \oplus \Z/2, \quad &k=3 \\
    0, \quad &\text{ else.}
    \end{cases}
\end{align*}
Since the differentials in the Atiyah--Hirzebruch spectral sequence increase the cohomological degree by at least $3$, 
the only differentials which start at a non-trivial cohomology group to a group with torsion are 
\begin{align*}
    d_3 \colon  H^0(PSO(4);\Z) &\longrightarrow H^3(PSO(4);\Z) \\
    d_3 \colon  H^2(PSO(4);\Z) &\longrightarrow H^5(PSO(4);\Z).
\end{align*}
However, since the Thom morphism is surjective in  degrees $\leq 2$, these differentials must also be trivial. 
Thus we can conclude that the Thom morphism is surjective for $PSO(4)$.

For the group $PSO(6)$, we also need to study the full cohomology ring. 
Since
\begin{align*}
    &\Hfree{\ast}{PSO(6)} \cong \Lambda(e_3,e_7,y_5) \\
    &H^\ast(PSO(6);\Z/2) \cong \Z[v]/(v^2) \otimes \Z/2[u_2,u_3,u_5]/(u_2^4,u_3^2,u_5^2),
\end{align*}
we get the Bockstein diagram
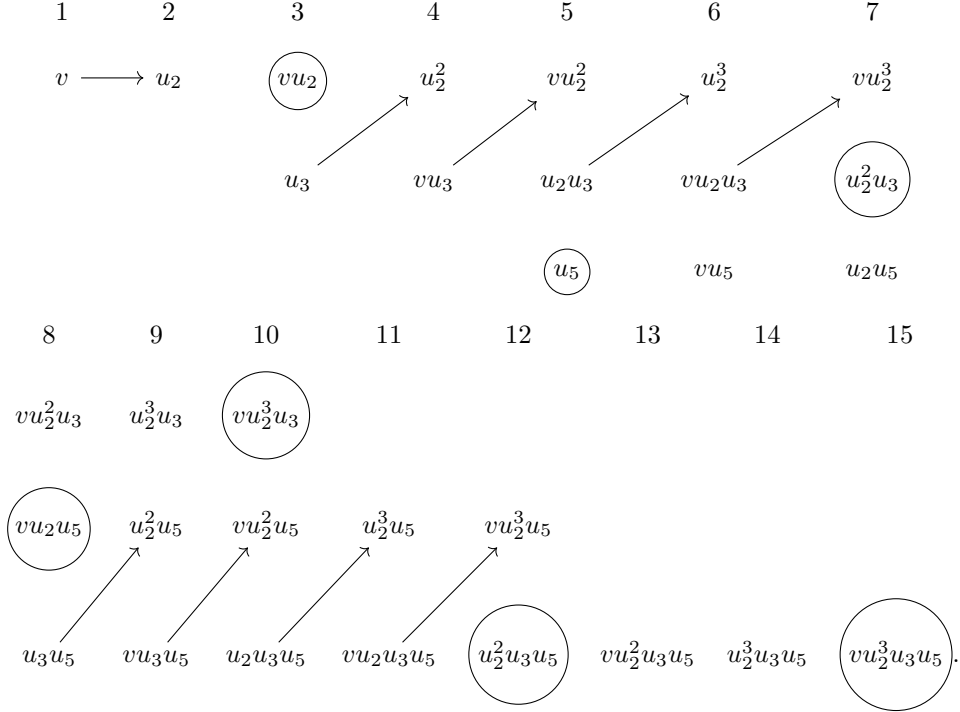
\begin{equation*}
\begin{tikzcd}[row sep=tiny]
 1 & 2 & 3 & 4 & 5 & 6 & 7\\
 v \arrow[r] & u_2 & \circled{$v u_2$} & u_2^2 & v u_2^2 & u_2^3 & v u_2^3 \\
 & & u_3 \arrow[ur] & v u_3 \arrow[ur] & u_2 u_3 \arrow[ur] & v u_2 u_3 \arrow[ur] & \circled{$u_2^2 u_3$} \\
 & & & & \circled{$u_5$} & v u_5 & u_2 u_5
\end{tikzcd}    
\end{equation*}
\begin{equation*}
\begin{tikzcd}[cramped, row sep=tiny, column sep=tiny]
 8 & 9 & 10 & 11 & 12 & 13 & 14 & 15\\
v u_2^2 u_3 & u_2^3 u_3 & \circled{$v u_2^3 u_3$} \\
\circled{$v u_2 u_5$} & u_2^2 u_5 & v u_2^2 u_5 & u_2^3 u_5 & v u_2^3 u_5 \\
u_3 u_5 \arrow[ur] & v u_3 u_5 \arrow[ur]& u_2 u_3 u_5 \arrow[ur] & v u_2 u_3 u_5 \arrow[ur] & \circled{$u_2^2 u_3 u_5$} & v u_2^2 u_3 u_5 & u_2^3 u_3 u_5 & \circled{$v u_2^3 u_3 u_5$}.
\end{tikzcd}    
\end{equation*}
The elements that correspond to non-torsion in $H^\ast(PSO(6);\Z)$ have been circled. This gives us the cohomology groups
\begin{align}\label{pso6}
    H^k(PSO(6);\Z) \cong \begin{cases}
        \Z, \quad & k=0,3,8,15 \\
        \Z/4, \quad & k=2,14 \\
        \Z/2, \quad & k=4,6,11 \\
        \Z \oplus \Z/2, \quad & k=5,10,12 \\
        \Z/4 \oplus \Z/2, \quad & k=9 \\
        \Z \oplus \Z/4 \oplus \Z/2, \quad & k=7\\
        0, \quad &\text{else}.
    \end{cases}
\end{align}
Since isomorphism \eqref{pso6} gives us all the elements in the image of the reduction homomorphism 
$H^\ast(PSO(6);\Z) \rightarrow H^\ast(PSO(6);\Z_2)$, 
it is straight-forward to check that all Steenrod operations of odd degree greater than $1$ are trivial for these elements. 
This shows that all differentials in the Atiyah--Hirzebruch spectral sequence are trivial. 
Thus the Thom morphism is surjective.  
\end{proof}



\subsection{Groups with Lie algebra \texorpdfstring{$\mathfrak{c_n}$}{Cn}}

The integral cohomology of the simply-connected Lie group $Sp(n)$ is given by
\begin{align*}
    H^\ast(Sp(n);\Z) \cong \Lambda(e_3,e_7,\ldots,e_{4n-1}).
\end{align*}
Since the cohomology is torsion-free, we can  conclude: 
\begin{prop}\label{prop:Spn}
The Thom morphism is surjective in all cohomological degrees for $Sp(n)$. \qed
\end{prop}
The center of $\mathrm{Sp}(n)$ is isomorphic to $\Z/2$, consisting of the positive and negative of the identity matrix. 
It follows that there is only one other compact Lie group with the same Lie algebra which we obtain by dividing out by $Z(\mathrm{Sp}(n))$ (see \cite{bb}). 
This group is known as \textit{the projective symplectic group}, denoted by $PSp(n)$.


The $\Z/2$-cohomology of $PSp(n)$ is given by 
\begin{align*}
    H^\ast(PSp(n);\Z/2) &\cong \Z/2[v]/(v^{4q}) \otimes \Lambda(b_3,b_7,\ldots,\widehat{b}_{4q-1},\ldots,b_{4n-1}),
\end{align*}
where $q$ is the largest power of $2$ dividing $n$. The Steenrod squares are given by
\begin{align*}
    \Sq^{4j}(b_{4k+3}) = \binom{k}{j} b_{4k+4j+3}
\end{align*}
with all other Steenrod squares trivial, except
\begin{align}\label{pspeven}
    \Sq^1(b_{2q-1}) = v^{2q}
\end{align}
when $n$ is even. 
Equation (\ref{pspeven}) makes the cases of even and odd $n$ different. 
We start with the even case.

\begin{prop}
For all even $n\ge 2$, the generator $e_{2q-1}\in H^{2q-1}(PSp(n);\Z)$ is not in the image of the Thom morphism. 
\end{prop}
\begin{proof}
Assume that $n \ge 2$ is even. 
We may consider the following part of the Bockstein diagram
\begin{equation*}
\begin{tikzcd}
 2q-2 & 2q-1 & 2q \\
 v^{2q-2} & v^{2q-1} \arrow[r] & v^{2q} \\
& b_{2q-1}. \arrow[ur]
\end{tikzcd}
\end{equation*}
It follows that the generator $e_{2q-1}$ is mapped to $v^{2q-1}+b_{2q-1}$ plus torsion. 
We have $\Sq^{2q-1}(v^{2q-1}+b_{2q-1}) = v^{4q-2} \neq 0$, and hence the assertion.  
\end{proof}

\begin{prop}
For all odd $n$, the Thom morphism is surjective for $PSp(n)$ in all cohomological degrees. 
\end{prop}
\begin{proof}
Assume that $n$ is odd.  
We have the cohomology ring
\begin{align*}
    H^\ast(PSp(n);\Z/2) \cong \Z/2[v]/(v^4) \otimes \Lambda(b_7,b_{11},\ldots,b_{4n-1}).
\end{align*}
The only non-trivial Bocksteins are the ones that go from a term containing $v$ to a term containing $v^2$. 
The Bockstein diagram then takes the form 
\begin{equation*}
\begin{tikzcd}[column sep=small]
    1 & 2 & 3 & \cdots & 7 & 8 & 9 & 10 & \cdots \\
    v \arrow[r] & v^2 & v^3 & & b_7 & v b_7 \arrow[r] & v^2 b_7 & v^3 b_7 & \cdots
\end{tikzcd}    
\end{equation*}
The non-torsion elements of $\Hfree{\ast}{PSp(n)}$ map to elements of the form $b_{i_1}\cdots b_{i_k}$ or $v^3b_{i_1}\cdots b_{i_k}$ in $H^\ast(PSp(n);\Z/2)$, while the torsion elements are sent to elements of the form $v^2b_{i_1}\cdots b_{i_k}$. 
We now claim that none of these elements can survive an odd-dimensional Steenrod square.

Assume that $\alpha \in H^\ast(PSp(n);\Z/2)$ is such that $\Sq^{2n+1}(\alpha) \neq 0$. Then
\begin{align*}
    \Sq^1 \Sq^{2n} \alpha \neq 0,
\end{align*}
which implies that $\Sq^{2n}\alpha$ is of the form $v b_{i_1}\cdots b_{i_k}$. 
Since the number of $v$'s cannot be changed by a Steenrod square of even degree, the class $\alpha$ is a product of $b_i$'s and precisely one $v$. 
However, as seen in the Bockstein diagram, such elements do not correspond to elements of the integral cohomology of $PSp(n)$. 
This implies the assertion. 
\end{proof}



\subsection{Groups with Lie algebra \texorpdfstring{$\mathfrak{a_n}$}{An}}

The integral cohomology of the simply-connected Lie group $SU(n)$ is given by
\begin{align*}
    H^\ast(SU(n);\Z) \cong \Lambda(e_3,e_5,\ldots,e_{2n-1}), 
\end{align*}
and is torsion-free. 
Hence we can conclude: 
\begin{prop}\label{prop:SUn}
The Thom morphism is surjective in all cohomological degrees for $SU(n)$. \qed
\end{prop}

The center of $SU(n)$ is isomorphic to $\Z/n$. 
Hence, depending on $n$, we can take several different quotients of this group. 
The case where we divide out be the entire center is known as the \textit{projective special unitary group} $PSU(n)$. 
The cohomology groups of the various quotients are as follows. 
Let $l$ be a natural number dividing $n$. 
Let $\Gamma_l$ be the subgroup of $Z(SU(n))$ of order $l$, and set $G(n,l):=SU(n)/\Gamma_l$. 
Suppose $p$ is an odd prime dividing $l$ and $p^r$ is the largest power of $p$ dividing $n$. 
Then
\begin{align*}
    H^\ast(G(n,l);\Z/p) \cong \Z/p[y]/(y^{p^r}) \otimes \Lambda(z_1,z_3,\ldots,\widehat{z}_{2p^r-1},\ldots,z_{2n-1}),
\end{align*}
where $\lvert y \rvert = 2$. The power operations are given by
\begin{align*}
    P^k(z_{2i-1}) = \binom{i-1}{k}z_{2i-1+2k(p-1)}, ~ \text{and} ~ 
    \beta(z_{2p^{r-1}-1}) = y^{p^{r-1}}.
\end{align*}
Similarly, if $p=2$ and $4 \mid l$, then
\begin{align*}
    H^\ast(G(n,l);\Z/2) \cong \Z/2[y]/(y^{2^r}) \otimes \Lambda(z_1,z_3,\ldots,\widehat{z}_{2^{r+1}-1},\ldots,z_{2n-1}),
\end{align*}
with
\begin{align*}
\Sq^{2k}(z_{2i-1}) = \binom{i-1}{k}z_{2i-1+2k}, ~ \text{and} ~  \Sq^1(z_{2^r-1}) = y^{2^{r-1}},
\end{align*}
and all other odd-degree Steenrod operations are trivial. 
Finally, if $\modu{l}{2}{4}$, then
\begin{align}\label{su}
    H^\ast(G(n,l);\Z/2) \cong \Z/2[z_1]/({z_1}^{2^{r+1}}) \otimes \Lambda(z_3,z_5,\ldots,\widehat{z}_{2^{r+1}-1},\ldots,z_{2n-1}),
\end{align}
with 
the Steenrod operations
\begin{align*}
\Sq^{2k}(z_{2i-1}) = \binom{i-1}{k}z_{2i-1+2k} ~ \text{and} ~ 
    \Sq^1(z_{2^r-1}) &= z_1^{2^{r}}.
\end{align*}
Although there is significant torsion in the cohomology of $G(n,l)$, the Thom morphism is only non-surjective in specific cases.

\begin{prop}
Let $n, l \ge 1$ be integers with $4 \mid n$ and $\modu{l}{2}{4}$. 
Then the generator $e_{2^r-1} \in H^{2^r-1}(G(n,l);\Z)$ is not in the image of the Thom morphism.
\end{prop}
\begin{proof}
The cohomology of $G(n,l)$ is as in \eqref{su}, with $r\geq 2$. 
The relevant part of the Bockstein diagram is
\begin{equation*}
\begin{tikzcd}
\text{Degree:} & 2^{r-1} & 2^r -1 & 2^r \\
\text{Generators:} & z_1^{2^{r-1}} & {z_1}^{2^r -1} \arrow[r] & {z_1}^{2^r} \\
& & z_{2^r -1} \arrow[ur]
\end{tikzcd}
\end{equation*}
It follows that $e_{2^r -1} \in H^{2^r-1}(G(n,l);\Z)$ maps to ${z_1}^{2^r -1} + z_{2^r -1}$ (possibly plus torsion), for which
\begin{align*}
    \Sq^{2^r - 1}({z_1}^{2^r -1} + z_{2^r -1}) = {z_1}^{2^{r+1}-2} \neq 0.
\end{align*}
Thus, $e_{2^r-1}$ is not in the image of the Thom morphism.
\end{proof}
\begin{prop}
    Let $n$, $l$ be positive integers such that $4 \nmid n$ or $\nmodu{l}{2}{4}$. Then the Thom morphism is surjective in all cohomological degrees for $G(n,l)$.
\end{prop}
\begin{proof}
We will show that all the differentials starting in the top row in the Atiyah--Hirzebruch spectral sequence for $G(n,l)$ must be trivial. 
Suppose a sequence of power operations and Bocksteins of odd degree $\geq\! 3$ is non-trivial when evaluated on an element of $H^\ast(G(n,l);\Z/p)$, where $p$ is an odd prime. 
Since the only non-trivial Bockstein is 
\begin{align*}
    \beta(z_{2p^{r-1}-1}) &= y^{p^{r-1}},
\end{align*}
this can only occur if there is some $z_{2i-1}$ such that 
\begin{align} \label{oddtorsioneq}
    P^k(z_{2i-1}) = \binom{i-1}{k} z_{2i-1+2k(p-1)} = z_{{2p^{r-1}-1}}
\end{align}
for some $k$. 
We will now show that this is impossible by showing that for any $i,k,r$ satisfying \eqref{oddtorsioneq}, 
the binomial coefficient is divisible by $p$.

To simplify the notation, let $j=i-1$. The binomial coefficient $\binom{j}{k}$ can be computed using Lucas' theorem, which says that if 
\begin{align*}
    j &= j_0 + j_1 p + \ldots j_m p^m \\
    k &= k_0 + k_1 p + \ldots k_m p^m
\end{align*}
are the base $p$ expansions of $j$ and $k$, then
\begin{align*}
    \binom{j}{k} \equiv \prod_{t=0}^m \binom{j_t}{k_t}\quad (\text{mod } p) .
\end{align*}
Here we use the convention that $\binom{j}{k} = 0$ if $j <k$. 
From equation \eqref{oddtorsioneq} we see that
\begin{align}\label{oddtorsioneq2}
    j = p^{r-1} + k - kp - 1.
\end{align}
We can assume that $r > 1$, since otherwise the only non-trivial Bockstein is $\beta(z_1) = y$, which leads to a surjective Thom morphism. 
Now, let $s$ be the smallest natural number such that $p^{s-1} \mid k$, but $p^s \nmid k$. 
From equation \eqref{oddtorsioneq2} we see that $s < r$, since otherwise $j$ would be negative. 
We then get
\begin{align*}
    j \equiv k - 1 \quad (\text{mod } p^s).
\end{align*}
Since $k \not\equiv 0 \; (\text{mod } p^s)$, we see that $j_v < k_v$ for some $v < s$, and it follows from Lucas' theorem that $\modu{\binom{j}{k}}{0}{p}$. 
This proves that there are no non-trivial differentials of odd torsion.

It remains to show that there are no nontrivial differentials of $2$-torsion in the remaining cases. 
There are two such cases: when $n$ and $l$ are both multiples of $4$ and when $\modu{n}{2}{4}$. 
In the former case, the only non-trivial Bockstein is
\begin{align*}
    \Sq^1(z_{2^r-1}) &= y^{2^{r-1}},
\end{align*}
and the surjectivity of the Thom morphism follows from the same argument as with the odd torsion, by setting $P^k = \Sq^{2k}$.
In the latter case, if $\modu{n}{2}{4}$ and $\modu{l}{2}{4}$, the cohomology ring is
\begin{align*}
    H^\ast(G(n,l);\Z/2) \cong \Z/2[z_1]/(z_1^4) \otimes \Lambda(z_5,z_7,\ldots,z_{2n-1}),
\end{align*}
and the only Bockstein is $\Sq^1(z_1) = z_1^2$. The surjectivity of the Thom morphism then follows from the same argument as in the case $PSp(n)$ with $n$ odd. 
\end{proof}


\subsection{Groups with exceptional Lie algebras}\label{sec:exceptional_groups}

We will now consider Lie groups with exceptional Lie algebras. 
It turns out that the cases $\mathfrak{g_2}$, $mathfrak{f_4}$ and $\mathfrak{e_6}$ follow the same pattern, 
while we can say a bit more on $\mathfrak{e_7}$ and $\mathfrak{e_8}$. 
We will therefore split our analysis into three subsections. 

\subsubsection{Groups with Lie algebra \texorpdfstring{$\mathfrak{g_2}$}{G2},  \texorpdfstring{$\mathfrak{f_4}$}{F4} and \texorpdfstring{$\mathfrak{e_6}$}{E6}.}

The free cohomologies of the exceptional Lie groups $G_2$, $F_4$ and $E_6$ are given by
\begin{align*}
    H_{\text{free}}^\ast(G_2;\Z) & \cong \Lambda(e_3,e_{11}), \\
    \Hfree{\ast}{F_4} & \cong \Lambda(e_3,e_{11},e_{15},e_{23}), \\
     \Hfree{\ast}{E_6} & \cong \Lambda(e_3,e_9,e_{11},e_{15},e_{17},e_{23}).
\end{align*}

The center of $E_6$ is isomorphic to $\Z/3$ (see \cite{mt}). 
Hence there is another group which has Lie algebra $\mathfrak{e_6}$. 
We will refer to this group as the \textit{centerless $E_6$} and denote it by $E_6/\Gamma_3$. 
%

\begin{prop}
For $G_2$, $F_4$, $E_6$ and $E_6/\Gamma_3$, the generator $e_3$ in integral cohomology is not in the image of the Thom morphism. 
\end{prop}
\begin{proof}
The $\Z/2$-cohomologies of $G_2$, $F_4$ and $E_6$ are given by 
\begin{align*}
    H^\ast(G_2; \Z/2) & \cong \Z/2[x_3]/(x_3^4) \otimes \Lambda(x_5), \\
    H^\ast(F_4; \Z/2) & \cong \Z/2[x_3]/(x_3^4) \otimes \Lambda(x_5,x_{15},x_{23}), \\ 
     H^\ast(E_6; \Z/2) & \cong \Z/2[x_3]/(x_3^4) \otimes \Lambda(x_5,x_9,x_{15},x_{17},x_{23}).
\end{align*}
In each case, the generator $e_3$ in integral cohomology reduces to $x_3$ in $\Z/2$-cohomology and $\Sq^3(x_3) = x_3^2 \neq 0$. 
Hence $e_3$ is not in the image of the Thom morphism. 
Since neither the free cohomology nor the cohomology with coefficients in $\Z/2$ are altered by dividing out by a subgroup of order $3$, the same argument as for $E_6$ applies to $E_6/\Gamma_3$.    
\end{proof}

\begin{remark}
We note that the other generators in the integral cohomology groups of $G_2$, $F_4$, $E_6$ and   $E_6/\Gamma_3$ are in the image the Thom morphism. 
As we will see in proposition \ref{prop:E7} and \ref{prop:E8}, the behaviors of the cohomology of the groups corresponding to the Lie algebras $\mathfrak{e_7}$ and $\mathfrak{e}_8$ are  different. 
\end{remark}

\begin{remark}
The integral cohomology of the exceptional groups, except for $G_2$, also have $3$-torsion, and we could have used a computation at $p=3$ to show non-surjectivity. 
\end{remark}


\subsubsection{Groups with Lie algebra  \texorpdfstring{$\mathfrak{e_7}$}{E7}.}

The free cohomology of the group $E_7$ is given by
\begin{align*}
    \Hfree{\ast}{E_7} & \cong \Lambda(e_3,e_{11},e_{15},e_{19},e_{23},e_{27},e_{35}).     
\end{align*}
The center of $E_7$ is isomorphic to $\Z/2$ (see \cite{mt}), 
 and hence there is another group which has the Lie algebra $E_7$. 
We will refer to this group as the \textit{centerless $E_7$} and denote it by $E_7/\Gamma_2$.
\begin{prop}\label{prop:E7}
For $E_7$ and $E_7/\Gamma_2$, the generators $e_3$ and $e_{15}$ in integral cohomology are not in the image of the Thom morphism. 
\end{prop}
\begin{proof}
While there is significant $2$-torsion in the cohomology of $E_7$, 
it is more convenient to show the non-surjectivity of the Thom morphism in degree $3$ using $3$-torsion. 
The relevant cohomology ring is
\begin{align*}
    H^\ast(E_7;\Z/3) \cong\; &\Z/3[x_8]/(x_8^3) \otimes \Lambda(x_3,x_7,x_{11},x_{15},x_{19},x_{27},x_{35}), \\
    \text{ with } &P^1 x_3 = x_7,\; P^3 x_7 = x_{19},\; \beta x_7 = x_8.
\end{align*}
The reduction homomorphism $\rho \colon H^3(E_7;\Z) \rightarrow H^3(E_7;\Z/3)$ sends $e_3$ to $x_3$. 
Let $Q_1$ denote the first Milnor operation. 
We then have
\begin{align*}
    Q_1(x_3) = P^1 \beta (x_3) - \beta P^1 (x_3) = P^1(0) - \beta(x_7) = -x_8,
\end{align*}
and we conclude that $e_3$ is not in the image of $\tau$ by lemma \ref{lem:Milnor_obstruction}. 
Since dividing out by a subgroup of order $2$ changes neither the free cohomology nor the $\Z/3$-cohomology, we deduce that $e_3 \in H^3(E_7/\Gamma_2;\Z)$ is not in the image of $\tau$ either. 

To see that the generator $e_{15}$ is not in the image of $\tau$ we will use $2$-torsion. 
The $\Z/2$-cohomology and Steenrod operations of $E_7$ and $E_7/\Gamma_2$ are given by
\begin{align}\label{E7cohomology}
    H^\ast(E_7; \Z/2) \cong & \;\Z/2[x_3,x_5,x_9]/(x_3^{4},x_5^4,x_9^4) \otimes \Lambda(x_{15},x_{17},x_{23},x_{27}), \\
    H^\ast(E_7/\Gamma_2;\Z/2) \cong & \;\Z/2[x_1,x_5,x_9]/(x_1^{4},x_5^4,x_9^4) \otimes \Lambda(x_6,x_{15},x_{17},x_{23},x_{27}), \nonumber\\
    \text{ where } &\Sq^1 x_1 = x_1^2,\;\Sq^1 x_5=x_3^2,\; \Sq^1 x_9=x_5^2,\; \Sq^1 x_{15}=x_3^2x_5^2,\;  \nonumber \\
    &\Sq^1 x_{17}=x_9^2,\;\Sq^1 x_{23}=x_3^2x_9^2,\; \Sq^1 x_{27}=x_5^2x_9^2,\nonumber \\
    &\Sq^2 x_3=x_5,\; \Sq^2 x_{15}=x_{17},\nonumber \\
    & \Sq^4 x_5=x_9,\; \Sq^4 x_{23}=x_{27},\nonumber\\
    & \Sq^8 x_9=x_{17},\; \Sq^8 x_{15}=x_{23},\nonumber
\end{align}
with all other $\Sq^1, \Sq^2, \Sq^4, \Sq^8$ trivial and where $x_3^2$ is interpreted as $x_6$ in $H^\ast(E_7/\Gamma_2;\Z/2)$. In low degrees, we get the following Bockstein diagram for $E_7$:
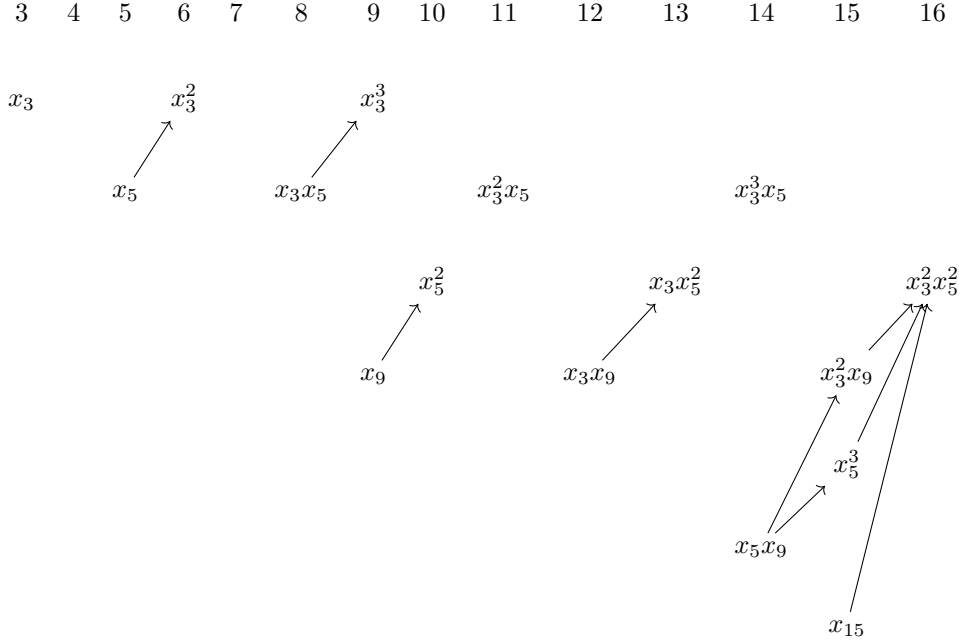
\begin{equation*}
\begin{tikzcd}[cramped, column sep=tiny]
3 & 4 & 5 & 6 & 7 & 8 & 9 & 10 & 11 & 12 & 13 & 14 & 15 & 16\\
 x_3 & & & |[alias=32]| x_3^2 & & & |[alias=33]| x_3^3 & & &  & & &  \\
 & & x_5 \arrow[to=32] & & & x_3 x_5 \arrow[to=33] & & & x_3^2 x_5  & & & |[alias=3351]| x_3^3 x_5  \\
 & & & & & & & |[alias=52]| x_5^2 & & & |[alias=3152]| x_3 x_5^2 & & & |[alias=3252]| x_3^2 x_5^2 \\
 & & & & & & x_9 \arrow[to=52] & & & x_3 x_9 \arrow[to=3152] & & & |[alias=3291]| x_3^2 x_9 \arrow[to=3252] \\
 & & & & & & & & & & & & |[alias=53]| x_5^3 \arrow[to=3252] \\
 & & & & & & & & & & & |[alias=5191]| x_5 x_9 \arrow[to=53] \arrow[to=3291] \\
 & & & & & & & & & & & & x_{15} \arrow[to=3252]
 \end{tikzcd}
 \end{equation*}
 The generator $e_{15}\in H^{15}(E_7,\Z)$ reduces to $x_{15} + x_3^2 x_9$ or $x_{15} + x_5^3$ in $H^{15}(E_7;\Z_2)$, and we compute
 \begin{align*}
     \Sq^3(x_{15} + x_3^2 x_9) = \Sq^3(x_{15} + x_5^3) = x_9^2 \neq 0 
 \end{align*}
to see that $e_{15} \in H^{15}(E_7;\Z)$ is not in the image of the Thom morphism.

For $E_7/\Gamma_2$, we have the following Bockstein diagram:
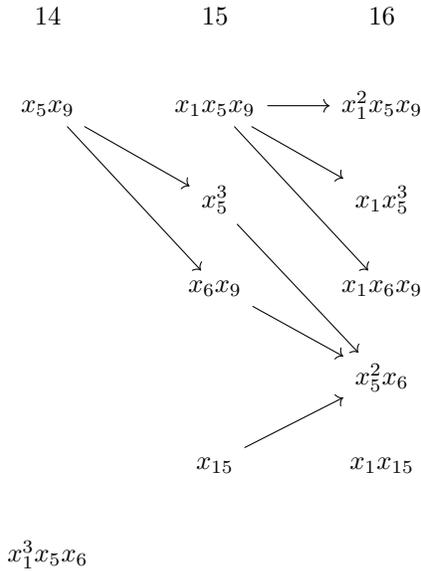
\begin{equation*}
\begin{tikzcd}
 14 & 15 & 16 \\
 x_5 x_9 \arrow[to=53] \arrow[to=69] & x_1 x_5 x_9 \arrow[to=153] \arrow[to=1259] \arrow[to=169] & |[alias=1259]| x_1^2 x_5 x_9 \\
 & |[alias=53]| x_5^3 \arrow[to=526] & |[alias=153]| x_1 x_5^3 \\
 & |[alias=69]| x_6 x_9 \arrow[to=526] & |[alias=169]| x_1 x_6 x_9 \\
 & & |[alias=526]| x_5^2 x_6 \\
 & x_{15} \arrow[to=526] & x_1 x_{15} \\
 x_1^3 x_5 x_6 
 \end{tikzcd}
 \end{equation*}
 It follows that $e_{15}\in H^{15}(E_7/\Gamma_2;\Z)$ reduces to either $x_{15} + x_6 x_9$ or $x_{15} + x_5^3$ in $H^{15}(E_7/\Gamma_2;\Z/2)$. 
Since
\begin{align*}
\Sq^3(x_{15} + x_6 x_9) = \Sq^3(x_{15} + x_5^3) = x_9^2 \ne 0,   
\end{align*}
we conclude that $e_{15}\in H^{15}(E_7/\Gamma_2;\Z)$ and hence the assertion.   
\end{proof}


\subsubsection{The group \texorpdfstring{$E_8$}{E8}}

The free cohomology of the group $E_8$ is given by
\begin{align*}
\Hfree{\ast}{E_8} \cong \Lambda(e_3,e_{15},e_{23},e_{27},e_{35},e_{39},e_{47},e_{59}). 
\end{align*}

\begin{prop}\label{prop:E8}
The generators $e_3, e_{15}, e_{23}, e_{27} \in \Hfree{\ast}{E_8}$ as well as the sum of any of these generators with a torsion class in the same degree are not in the image of the Thom morphism. 
\end{prop}
\begin{proof}
The $\Z/2$-cohomology of $E_8$ and the Steenrod operations are given by 
\begin{align}\label{E8cohomology}
    H^\ast(E_8; \Z/2) \cong & \;\Z/2[x_3,x_5,x_9,x_{15}]/(x_3^{16},x_5^8,x_9^4,x_{15}^4) \otimes \Lambda(x_{17},x_{23},x_{27},x_{29}), \\
    \text{ where } &\Sq^1 x_5=x_3^2,\; \Sq^1 x_9=x_5^2,\; \Sq^1 x_{15}=x_3^2x_5^2,\; \Sq^1 x_{17}=x_9^2,\nonumber \\
    &\Sq^1 x_{23}=x_3^2x_9^2,\; \Sq^1 x_{27}=x_5^2x_9^2,\; \Sq^1 x_{29}=x_{15}^2,\nonumber \\
    &\Sq^2 x_3=x_5,\; \Sq^2 x_{15}=x_{17},\; \Sq^2 x_{27}=x_{29},\nonumber \\
    & \Sq^4 x_5=x_9,\; \Sq^4 x_{23}=x_{27},\nonumber\\
    & \Sq^8 x_9=x_{17},\; \Sq^8 x_{15}=x_{23},\nonumber
\end{align}
with all other $\Sq^1, \Sq^2, \Sq^4, \Sq^8$ trivial.

In low degrees, we get the following Bockstein diagram for $p=2$:
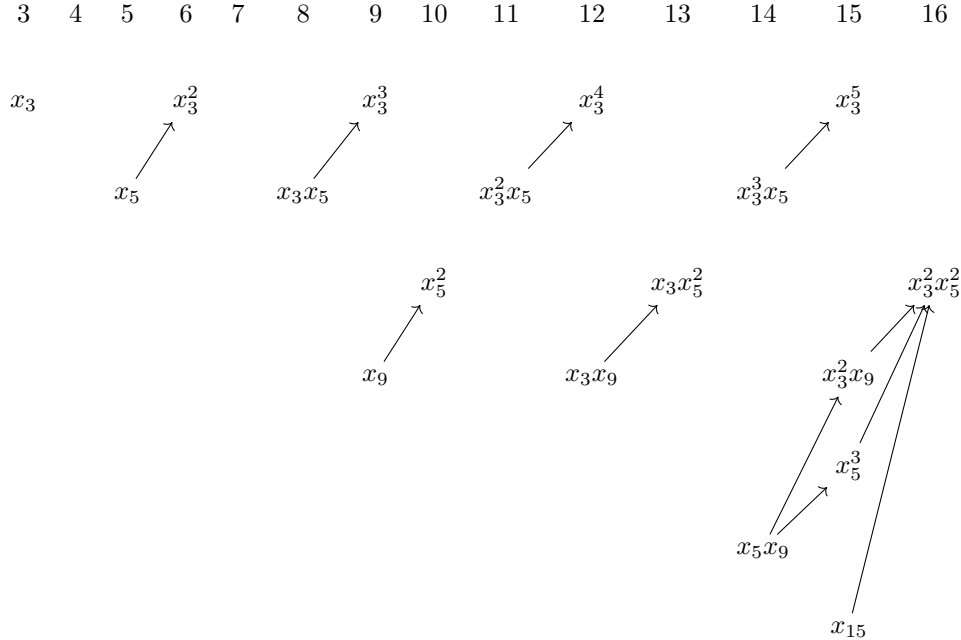
\begin{equation*}
\begin{tikzcd}[cramped, column sep=tiny]
3 & 4 & 5 & 6 & 7 & 8 & 9 & 10 & 11 & 12 & 13 & 14 & 15 & 16\\
 x_3 & & & |[alias=32]| x_3^2 & & & |[alias=33]| x_3^3 & & & |[alias=34]| x_3^4 & & & |[alias=35]| x_3^5 \\
 & & x_5 \arrow[to=32] & & & x_3 x_5 \arrow[to=33] & & & x_3^2 x_5 \arrow[to=34] & & & |[alias=3351]| x_3^3 x_5 \arrow[to=35] \\
 & & & & & & & |[alias=52]| x_5^2 & & & |[alias=3152]| x_3 x_5^2 & & & |[alias=3252]| x_3^2 x_5^2 \\
 & & & & & & x_9 \arrow[to=52] & & & x_3 x_9 \arrow[to=3152] & & & |[alias=3291]| x_3^2 x_9 \arrow[to=3252] \\
 & & & & & & & & & & & & |[alias=53]| x_5^3 \arrow[to=3252] \\
 & & & & & & & & & & & |[alias=5191]| x_5 x_9 \arrow[to=53] \arrow[to=3291] \\
 & & & & & & & & & & & & x_{15} \arrow[to=3252]
 \end{tikzcd}
 \end{equation*}
 
 \begin{equation} \label{E8bockstein2}
 \adjustbox{scale=0.75}{
 \begin{tikzcd}[column sep=0.3em, row sep=small]
  17 & 18 & 19 & 20 & 21 & 22 & 23 & 24 & 25 & 26 & 27 & 28 \\
 & |[alias=36]| x_3^6 & & & |[alias=37]| x_3^7 & & & |[alias=38]| x_3^8& & &  |[alias=39]|x_3^9 \\
 x_3^4 x_5 \arrow[to=36] & & & x_3^5 x_5 \arrow[to=37] & & & x_3^6 x_5 \arrow[to=38] & & & x_3^7 x_5 \arrow[to=39] \\
  & |[alias=3391]| x_3^3 x_9 \arrow[to=3352] & & & |[alias=3491]| x_3^4 x_9 \arrow[to=3452] & & & |[alias=3591]| x_3^5 x_9 \arrow[to=3552] & & & |[alias=3691]| x_3^6 x_9 \arrow[to=3652] \\
 x_3 x_5 x_9 \arrow[to=3153] \arrow[to=3391] & & & x_3^2 x_5 x_9 \arrow[to=3253] \arrow[to=3491] & & & x_3^3 x_5 x_9 \arrow[to=3353] \arrow[to=3591] & & & x_3^4 x_5 x_9 \arrow[to=3453] \arrow[to=3691] \\
 & & |[alias=3352]| x_3^3 x_5^2 & & & |[alias=3452]| x_3^4 x_5^2 & & & |[alias=3552]| x_3^5 x_5^2 & & & |[alias=3652]| x_3^6 x_5^2 \\
 & |[alias=3153]| x_3 x_5^3 \arrow[to=3352] & & & |[alias=3253]| x_3^2 x_5^3 \arrow[to=3452] & & & |[alias=3353]| x_3^3 x_5^3 \arrow[to=3552] & & & |[alias=3453]| x_3^4 x_5^3 \arrow[to=3652] \\
 & x_3 x_{15} \arrow[to=3352] & & & |[alias=3215]| x_3^2 x_{15} \arrow[to=3452] & & & |[alias=3315]| x_3^3 x_{15} \arrow[to=3552] & & & |[alias=3415]| x_3^4 x_{15} \arrow[to=3652] \\
& & & x_5 x_{15} \arrow[to=3253] \arrow[to=3215] & & & x_3 x_5 x_{15} \arrow[to=3353] \arrow[to=3315] & & & x_3^2 x_5 x_{15} \arrow[to=3453] \arrow[to=3415] \\
& & & & & & & & |[alias=5215]| x_5^2 x_{15} \arrow[to=3254] & & & |[alias=315215]| x_3 x_5^2 x_{15} \\
& & & & & & & x_9 x_{15} \arrow[to=325291] \arrow[to=5215] & & & x_3 x_9 x_{15} \arrow[to=335291] \arrow[to=315215] \\
 & & & |[alias=54]| x_5^4 & & & |[alias=3154]| x_3 x_5^4 & & & |[alias=3254]| x_3^2 x_5^4 \\
 & & x_5^2 x_9 \arrow[to=54] & & & x_3 x_5^2 x_9 \arrow[to=3154] & & & |[alias=325291]| x_3^2 x_5^2 x_9 \arrow[to=3254] & & & |[alias=335291]| x_3^3 x_5^2 x_9\\
 & & & & & & & & |[alias=55]| x_5^5 \arrow[to=3254] & & & |[alias=3155]| x_3 x_5^5 \\
 & & & & & & & x_3^5 x_9 \arrow[to=55] \arrow[to=325291] & & & x_3 x_5^3 x_9 \arrow[to=3155] \arrow[to=335291] \\
& |[alias=92]| x_9^2 & & & |[alias=3192]| x_3 x_9^2 & & & |[alias=3292]| x_3^2 x_9^2 & & & |[alias=3392]| x_3^3 x_9^2 \\
 & & & & & & |[alias=5192]| x_5 x_9^2 \arrow[to=3292] & & & |[alias=315192]| x_3 x_5 x_9^2 \arrow[to=3392] \\
x_{17} \arrow[to=92] & & & x_3 x_{17} \arrow[to=3192] & & & |[alias=3217]| x_3^2 x_{17} \arrow[to=3292] & & & |[alias=3317]| x_3^3 x_{17} \arrow[to=3392] \\
& & & & & x_5 x_{17} \arrow[to=5192] \arrow[to=3217] & & & x_3 x_5 x_{17} \arrow[to=315192] \arrow[to=3317] & & & x_3^2 x_5 x_{17} \\
& & & & & & x_{23} \arrow[to=3292] & & & x_3 x_{23} \arrow[to=3392] \\
& & & & & & & & & & & |[alias=5292]| x_5^2 x_9^2 \\
& & & & & & & & & & |[alias=5217]| x_5^2 x_{17} \arrow[to=5292] \\
& & & & & & & & & & |[alias=93]| x_9^3 \arrow[to=5292] \\
& & & & & & & & & x_9 x_{17} \arrow[to=93] \arrow[to=5217] \\
& & & & & & & & & & & x_5 x_{23} \\
& & & & & & & & & & x_{27} \arrow[to=5292]
\end{tikzcd}
}
\end{equation}
The class $e_3\in \Hfree{3}{E_8}$ reduces to the class $x_3 \in H^3(E_8;\Z/2)$, for which $\Sq^3$ is non-trivial. For the other degrees, we start in degree $27$ and work our way down.

From diagram \eqref{E8bockstein2}, we see that under the reduction map $H^{27}(E_8;\Z) \rightarrow H^{27}(E_8;\Z/2)$, the class $e_{27}$ + (torsion) is sent to either
\begin{align} \label{E827reductions}
    x_{27} + x_5^2 x_{17} + L \quad \text{ or }\quad    x_{27} + x_9^3 + L,
\end{align}
where $L$ is some linear combination of the classes
\begin{align*}
    x_3^3 x_9^2, ~ x_3^9, ~ x_3^6 x_9 + x_3^4 x_5^3, ~  \text{and} ~ x_3^4 x_{15} + x_3^4 x_5^3.
\end{align*}
We observe that
\begin{align*}
    \Sq^3(x_{27}) = \Sq^1\Sq^2(x_{27}) = \Sq^1(x_{29})=x_{15}^2.
\end{align*}
Then, using \eqref{E8cohomology}, we see that none of the other terms in \eqref{E827reductions} can yield an $x_{15}^2$-term when applying $\Sq^3$. 
Thus, the term $x_{15}^2$ cannot get cancelled out, and it follows that $e_7$ and $e_7$ plus torsion are not in the image of the Thom morphism.

For the class $e_{23}$, we get four alternatives for its reduction to $\Z/2$-cohomology:
\begin{align}\label{E823mod2reductions}
x_{23} + x_5 x_9^2, ~ 
x_{23} + x_5 x_9^2 + x_3 x_5^4, ~ 
x_{23} + x_3^2 x_{17}, ~ \text{and} ~  
x_{23} + x_3^2 x_{17} + x_3 x_5^4. 
\end{align}
We evaluate $\Sq^4$ on each of these classes and get
\begin{align*}
    &\Sq^4(x_{23} + x_5 x_9^2) = x_{27} + x_9^3 \\
    &\Sq^4(x_{23} + x_5 x_9^2 + x_3 x_5^4) = x_{27} + x_9^3 + x_3^9 \nonumber \\
    &\Sq^4(x_{23} + x_3^2 x_{17}) = x_{27} + x_5^2 x_{17} \nonumber \\
    &\Sq^4(x_{23} + x_3^2 x_{17} + x_3 x_5^4) = x_{27} + x_5^2 x_{17} + x_3^9. \nonumber
\end{align*}
We see that each cohomology class in \eqref{E823mod2reductions} is mapped to a class in \eqref{E827reductions} by $\Sq^4$. 
It follows that $e_{23}$ reduces to a class for which $\Sq^1 \Sq^2 \Sq^4$ is nonzero. 
This shows that $e_{23}$ plus any torsion is not in the image of the Thom morphism.

Finally, there are four possibilities for the mod-$2$ reduction of $e_{15}$ plus torsion, given by 
\begin{align}\label{E815reductions}
x_{15} + x_3^2 x_9, ~ 
x_{15} + x_3^2 x_9 + x_3^5, ~ 
x_{15} + x_5^3, ~ \text{and} ~ 
x_{15} + x_5^3 + x_3^5. 
\end{align}

The operation $\Sq^8$ acts on these classes by 
\begin{align*}
    \Sq^8(x_{15} + x_3^2 x_9) &= x_{23} + x_3^2 x_{17} \\
    \Sq^8(x_{15} + x_3^2 x_9 + x_3^5) &= x_{23} + x_3^2 x_{17} + x_3 x_5^4 \nonumber \\
    \Sq^8(x_{15} + x_5^3) &= x_{23} + x_5 x_9^2 \nonumber \\
    \Sq^8(x_{15} + x_5^3 + x_3^5) &= x_{23} + x_5 x_9^2 + x_3 x_5^4. \nonumber
\end{align*}
Hence all classes in \eqref{E815reductions} have a nonzero image under $\Sq^1 \Sq^2 \Sq^4 \Sq^8$, proving that $e_{15}$ plus torsion is not in the image of the Thom morphism.
\end{proof}


\section{Geometric examples for special orthogonal groups}\label{sec:geometric_examples}

Recall from proposition \ref{prop:SOn_for_n_ge_5} that the generator $e_3 \in \Hfree{\ast}{SO(5)} \cong \Lambda(e_3,e_7)$ is not in the image of $\tau$. 
However, we will now show that the element $2e_3$ is in the image of the Thom morphism.   
We will prove this by geometrically constructing an element of $MU^3(SO(5))$ which is mapped to $2e_3 \in H^3(SO(5);\Z)$ under $\tau$. 

To do so we make use of the fact that $SO(5)$ is a $10$-dimensional compact manifold. 
Let $2\tilde{e}_3 \in H_7(SO(5);\Z)$ denote the image of $2e_3$ under the isomorphism $H^3(SO(5);\Z) \cong H_7(SO(5);\Z)$ defined by Poincar\'e duality. 
By \cite[Proposition 1.2]{quillen}, elements in $MU^3(SO(5))$ can be represented by proper complex-oriented maps of the form $M \to SO(5)$ where $M$ is a $7$-dimensional manifold. 
Thus, in order to show that $2e_3$ is in the image of $\tau$, it suffices to find a proper complex-oriented map $g \colon M \to SO(5)$ such that $g_*[M] = 2\tilde{e}_3$ where $[M]$ denotes the fundamental class of $M$ in $H_7(SO(5);\Z)$. 

We will therefore now compute the homology group $H_7(SO(5);\Z)$. 
We will do this using an explicit cell structure of $SO(5)$. 


\subsection{The cell structure of special orthogonal groups}\label{sec:homology_of_SOn}

We recall the cell structure of special orthogonal groups using maps from products of real projective spaces from \cite[Proposition 3D.1]{hatcher}.  
Let $v$ be a nonzero vector in $\R^n$. 
We define the linear transformation $r(v)\colon \R^{n} \rightarrow \R^{n}$ to be the reflection across the orthogonal complement of $v$. 
We may use this  map to define an embedding from $\RP^{k-1}$ to $SO(n)$ for $k \leq n$ as follows.   
Representing elements of $\RP^{k-1}$ by vectors in $\R^k$ and embedding into $\R^n$ in the canonical way if $k < n$, we define
\begin{align*}
 \RP^{k-1} & \longrightarrow SO(n) \\
[v] & \longmapsto r(v) \cdot r(e_1)
\end{align*}
where $e_1,\ldots,e_n$ are the standard basis vectors. 
We extend this to a map defined on products of real projective spaces by taking compositions, i.e., 
\begin{align*}
f_{i_1,\ldots,i_m} \colon \RP^{i_1}\times \ldots \times \RP^{i_m} &\longrightarrow SO(n) \\
([v_1],\ldots,[v_m]) &\longmapsto r(v_1) \cdot r(e_1) \cdots r(v_m) \cdot r(e_1).
\end{align*}
For $SO(n)$, there is a $k$-cell for each sequence $(i_1,\ldots,i_m)$ which satisfies both  $n > i_1 > \ldots > i_m > 0$ and $i_1 + \ldots + i_m = k$. 
The characteristic map is given by
\begin{equation*}
\begin{tikzcd}
D^k \arrow[r, "\cong"] & D^{i_1} \times \ldots \times D^{i_m} \arrow[r] & \RP^{i_1} \times \ldots \times \RP^{i_m} \arrow[r] & SO(n),
\end{tikzcd}
\end{equation*}
where the second map is the product of the characteristic maps for the top cells of each real projective space. 
There is a single $0$-cell, namely the identity of $SO(n)$.

This gives us all the information we need to construct the cellular chain complex of $SO(5)$. The differentials are determined by the differentials in the cellular chain complexes of real projective spaces, as well as the product formula 
\begin{align*}
d(e^i \times e^j) = d(e^i)\times e^j +(-1)^i e^i \times d(e^j).
\end{align*}
This provides a complete description of the cellular chain complex of $SO(5)$. 
In the following diagram each node is a cell, where for example $(2,1)$ is the cell given by the map $f_{2,1}\colon \RP^2 \times \RP^1 \rightarrow SO(5)$. 
The line segments indicate when a cell is  
contained in another cell. 

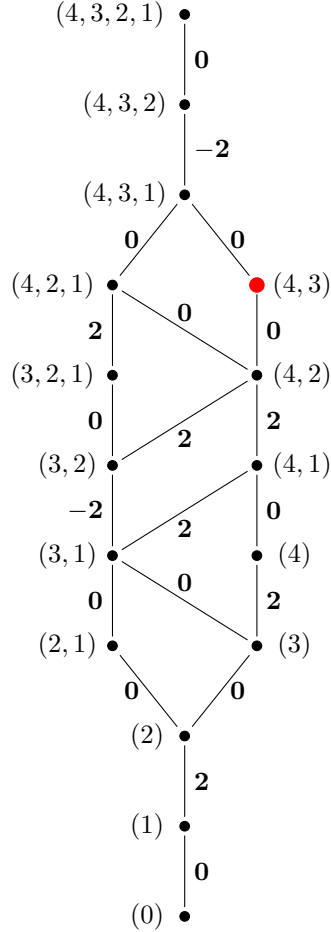
\begin{equation}\label{so5cw}
\begin{tikzpicture}[scale=0.8]


\node(0)at(0,0){};
\node(1)at(0,1.5){};
\node(2)at(0,3){};
\node(21)at(-1.2,4.5){};
\node(3)at(1.2,4.5){};
\node(31)at(-1.2,6){};
\node(4)at(1.2,6){};
\node(32)at(-1.2,7.5){};
\node(41)at(1.2,7.5){};
\node(321)at(-1.2,9){};
\node(42)at(1.2,9){};
\node(421)at(-1.2,10.5){};
\node(43)at(1.2,10.5){};
\node(431)at(0,12){};
\node(432)at(0,13.5){};
\node(4321)at(0,15){};


\draw[fill=black] (0)circle(0.08);
\draw[fill=black] (1)circle(0.08);
\draw[fill=black] (2)circle(0.08);
\draw[fill=black] (21)circle(0.08);
\draw[fill=black] (3)circle(0.08);
\draw[fill=black] (31)circle(0.08);
\draw[fill=black] (4)circle(0.08);
\draw[fill=black] (32)circle(0.08);
\draw[fill=black] (41)circle(0.08);
\draw[fill=black] (321)circle(0.08);
\draw[fill=black] (42)circle(0.08);
\draw[fill=black] (421)circle(0.08);
\draw[fill=red, draw=red] (43)circle(0.12);
\draw[fill=black] (431)circle(0.08);
\draw[fill=black] (432)circle(0.08);
\draw[fill=black] (4321)circle(0.08);


\node[xshift=-0.5cm] at (0) {$(0)$};
\node[xshift=-0.5cm] at (1) {$(1)$};
\node[xshift=-0.5cm] at (2) {$(2)$};
\node[xshift=-0.6cm] at (21) {$(2,1)$};
\node[xshift=0.5cm] at (3) {$(3)$};
\node[xshift=-0.6cm] at (31) {$(3,1)$};
\node[xshift=0.5cm] at (4) {$(4)$};
\node[xshift=-0.6cm] at (32) {$(3,2)$};
\node[xshift=0.6cm] at (41) {$(4,1)$};
\node[xshift=-0.8cm] at (321) {$(3,2,1)$};
\node[xshift=0.6cm] at (42) {$(4,2)$};
\node[xshift=-0.8cm] at (421) {$(4,2,1)$};
\node[xshift=0.6cm] at (43) {$(4,3)$};
\node[xshift=-0.8cm] at (431) {$(4,3,1)$};
\node[xshift=-0.8cm] at (432) {$(4,3,2)$};
\node[xshift=-1cm] at (4321) {$(4,3,2,1)$};


\draw (0) -- node[right]{$\mathbf{0}$} (1);
\draw (1) -- node[right]{$\mathbf{2}$} (2);
\draw (2) -- node[left]{$\mathbf{0}$} (21);
\draw (2) -- node[right]{$\mathbf{0}$} (3);
\draw (21) -- node[left]{$\mathbf{0}$} (31);
\draw (3) -- node[above]{$\mathbf{0}$} (31);
\draw (3) -- node[right]{$\mathbf{2}$} (4);
\draw (31) -- node[left]{$\mathbf{-2}$} (32);
\draw (31) -- node[below]{$\mathbf{2}$} (41);
\draw (4) -- node[right]{$\mathbf{0}$} (41);
\draw (32) -- node[left]{$\mathbf{0}$} (321);
\draw (32) -- node[below]{$\mathbf{2}$} (42);
\draw (41) -- node[right]{$\mathbf{2}$} (42);
\draw (321) -- node[left]{$\mathbf{2}$} (421);
\draw (42) -- node[above]{$\mathbf{0}$} (421);
\draw (42) -- node[right]{$\mathbf{0}$} (43);
\draw (421) -- node[left]{$\mathbf{0}$} (431);
\draw (43) -- node[right]{$\mathbf{0}$} (431);
\draw (431) -- node[right]{$\mathbf{-2}$} (432);
\draw (432) -- node[right]{$\mathbf{0}$} (4321);

\end{tikzpicture}
\end{equation}

\subsection{A geometric cobordism class on  \texorpdfstring{$SO(5)$}{SO5}}\label{sec:geometric_example_SO5}

From the structure of the complex displayed in \eqref{so5cw}, we see that $H_7(SO(5);\Z) \cong \Z$, where the generator comes from the cell $(4,3)$, coloured red in the diagram. 
However, since $\RP^4$ is not orientable, the map $f_{4,3}\colon \RP^4 \times \RP^3 \rightarrow SO(5)$ does not represent a bordism class. 
In fact, we know from the algebraic obstruction that the image of $e_3$ under $H_7(SO(5))$ cannot be hit by a bordism class on $SO(5)$. 
We will now show that we can replace $\RP^4 \times \RP^3$ with an orientable smooth manifold $M$ and the map $f_{4,3}$ by a smooth map $g \colon M \to SO(5)$ with the same image as $f_{4,3}$.

To do so, we first observe that the cell decomposition implies that every element of $SO(5)$ can be expressed as a composition of reflections in $\R^5$, where every pair of reflections leaves a $3$-dimensional subspace fixed and performs a rotation in the remaining $2$-plane. 
Let $\Gr{2}{5}$ be the Grassmann manifold of \emph{oriented} $2$-dimensional planes in $\R^5$. 
We will write elements of $\Gr{2}{5}$ in the form $(L,\sigma)$, where $L$ is a plane and $\sigma$ is an orientation of $L$. 
We then define
\begin{align*}
g \colon \Gr{2}{5} \times S^1 \longrightarrow SO(5)
\end{align*}
to be the map that sends $((L,\sigma),e^{it})$ to the element of $SO(5)$ which rotates the plane $L$ by the angle $t$ according to the orientation $\sigma$. 
More precisely, given a point $((L,\sigma),e^{it}) \in \Gr{2}{5} \times S^1$, let $r_{L,\sigma,t}$ be the rotation of $L$ by the angle $t$ along $\sigma$. 
Let $L^{\perp}$ denote the orthogonal complement of $L$ in $\R^5$ with the respect to the standard inner product. 
Then we can write $v \in \R^5$ in a unique way as $v=v_1+v_2$ such that $v_1 \in L$ and $v_2 \in L^{\perp}$. 
The transformation $g((L,\sigma),e^{it}) \in SO(5)$ is then defined by 
\begin{align*}
g((L,\sigma),e^{it})(v) = r_{L,\sigma,t}(v_1) + v_2. 
\end{align*}

The transformation $g((L,\sigma),e^{it}) \in SO(5)$ varies smoothly with $(L,\sigma)$ and $t$ in $\Gr{2}{5} \times S^1$. 

\begin{lemma}\label{lem:g_is_cobordism_class}
The map $g$ admits a complex orientation. 
In particular, $g$ is a proper complex-oriented smooth map and represents an element in $MU^3(SO(5))$.  
\end{lemma}
\begin{proof}
The fact that $g$ admits a complex orientation follows from the facts that $S^1$ is stably almost complex, $\Gr{2}{5} \cong SO(5)/(SO(2)\times SO(3))$ is almost complex,  
and $SO(5)$ is a compact Lie group.  
\end{proof}

\begin{lemma}\label{lem:image_og_f_and_g_are_equal}
The images of the maps $f_{4,3}$ and $g$ in $SO(5)$ are equal, i.e., the image of the map $g\colon \Gr{2}{5} \times S^1  \to SO(5)$ is the cell $(4,3)$. 
\end{lemma}
\begin{proof}
To simplify the notation we write $f=f_{4,3}$. 
We begin with showing that $\Imm f \subseteq \Imm g$. 
We let $(u,v)\in \RP^4 \times \RP^3$, and we will show that there exist two elements in $\Gr{2}{5}\times S^1$ which map to the element $f(u,v) \in SO(5)$. 
When showing this we will assume that $u$ and $v$ are both different from $\pm e_1$, since otherwise the argument is similar.

The $4$-planes $e_1^\perp$ and $u^\perp$ intersect on a $3$-dimensional subspace of $\R^5$ which remains fixed under the map $r(u) \cdot r(e_1)$. 
We will call this subspace $M_u$. 
Likewise, we let $M_v$ denote the $3$-dimensional subspace fixed by $r(v) \cdot r(e_1)$. 
One can observe that both $M_u$ and $M_v$ are contained in $\mathrm{Span}\{e_2,e_3,e_4,e_5\}$.

We first deal with the case where $u=v$. Then $M_u=M_v$, and it follows that $f(u,v)$ is a rotation in the plane $L = M_u^\perp$. 
For each of the two possible orientations of $L$, there is precisely one angle in $S^1$ which gives the rotation corresponding to $f(u,v)$, which shows that there are two elements of $\Gr{2}{5}\times S^1$ which map to $f(u,v)$.

On the other hand, if $u\neq v$, we get that $M_u \cap M_v$ is a $2$-dimensional subspace of $\mathrm{Span}\{e_2,e_3,e_4,e_5\}$. Let $N = (M_u \cap M_v)^\perp$. 
We observe that $f(u,v)\in SO(5)$ maps $N$ to $N$. Since $N$ is $3$-dimensional, a $1$-dimensional subspace of $N$ is left fixed by $f(u,v)$ (one way to see this is that every continuous map $S^2 \rightarrow S^2$ has a fixed point), and we call this line $T$. 
We have now seen that $f(u,v)$ leaves $T \oplus (M_u \cap M_v)$ fixed. 
The remaining $2$-dimensional subspace of $N$ is then our choice of $L$, in other words
\begin{align*}
    L := (T \oplus (M_u \cap M_v))^\perp.
\end{align*}
Having found the plane where the rotation takes place, we may combine orientations $\sigma$ and elements of $S^1$ as in the case $u=v$ to get the desired element of $SO(5)$. This proves that $\Imm f \subseteq \Imm g$.

We now show that $\Imm g \subseteq \Imm f$. 
Let 
\[
((L,\sigma),e^{it}) \in \Gr{2}{5}\times S^1.
\]
Our goal is to find vectors $u$ and $v$ such that $f(u,v) = g((L,\sigma),e^{it})$. 
We can first observe that $L \cap \R^4$ is at least $1$-dimensional. 
We then choose $v^\prime$ to be any unit vector in this intersection and note that $v^\prime$ represents a point in $\RP^3$. 
Next, we want to find a suitable $u \in \RP^4$ such that
\begin{align*}
    r(u) \cdot r(v^\prime) = g((L,\sigma),e^{it}).
\end{align*}
Clearly, $u$ must be in $L$, since then $L^\perp$ is fixed by both $r(u)$ and $r(v^\prime)$. 
Furthermore, the angle between $u$ and $v^\prime$ is uniquely determined by $t$ to yield the desired  rotation. 
(In fact, the angle must be $t/2$ or $t/2 + \pi$, depending on which representative we choose for the point in $\RP^4$.) 
With the exception of the cases $t=0$ and $t=\pi$, this leaves two options for $u$, which we choose between by making sure the rotation $r(u) \cdot r(v^\prime)$ goes in the right direction according to the orientation $\sigma$.
The composition $r(v^\prime)\cdot r(e_1)$ has a unique inverse, which is given by $r(v) \cdot r(e_1)$ for some vector $v$ in $\RP^3$.
We then have  
\begin{align*}
    f(u,v) & = r(u) \cdot r(e_1) \cdot r(v) \cdot r(e_1) \\
    & = r(u) \cdot r(e_1) \cdot [r(v^\prime) \cdot r(e_1)]^{-1} \\
    & = r(u) \cdot r(e_1) \cdot r(e_1) \cdot r(v^\prime) \\
    & = r(u) \cdot r(v^\prime) \\
    & = g((L,\sigma),e^{it}).
\end{align*}
This proves that $\Imm f = \Imm g$.
\end{proof}

We can now show the main result of this section. 

\begin{theorem}\label{thm:geomtetric_example_SO5}
The cobordism class represented by $g$ maps to $2e_3 \in H^3(SO(5);\Z)$ under the Thom morphism.
\end{theorem}
\begin{proof}
\begin{sloppypar}By Poincar\'e duality it suffices to show that the fundamental class of ${\Gr{2}{5}\times S^1}$ in homology is mapped to two times the generator $\tilde{e}_3\in H_7(SO(5);\Z)$.  
By lemma \ref{lem:image_og_f_and_g_are_equal}, 
the image of $g$ is the cell $(4,3)$. 
Hence we can consider $g$ as a map\end{sloppypar}
\begin{align*}
    \Gr{2}{5}\times S^1 \longrightarrow (4,3).
\end{align*}
Let
\begin{align*}
    q\colon (4,3) \longrightarrow (4,3)/(4,3)_6
\end{align*}
be the map that collapses the $6$-skeleton of the cell $(4,3)$.  
We then get the following commutative diagram in homology
\begin{equation*}
\begin{tikzcd}
    H_7(\Gr{2}{5} \times S^1;\Z) \arrow[r, "g_\ast"] \arrow[rd, "(q\circ g)_\ast"'] & H_7((4,3);\Z) \arrow[d, "\cong"', "q_\ast"] \\
     & H_7((4,3)/(4,3)_6;\Z).
\end{tikzcd}
\end{equation*}
Using the homology long exact sequence of the pair $((4,3)_6,(4,3))$, it is straight-forward to see that the map $q_\ast$ is an isomorphism. 
The quotient $(4,3)/(4,3)_6$ is homeomorphic to $S^7$. 
Hence by choosing an orientation we can assume that   $q\circ g$ is a map between compact, oriented topological manifolds. 
Proving that $g_\ast$ is a multiplication by $\pm 2$ is hence reduced to the claim that the map $(q \circ g)_*$ has degree $\pm 2$. 
We will show this claim by computing the local degree of $q \circ g$ at two points in $\Gr{2}{5}\times S^1$.

Let $y\in SO(5)$ be the point corresponding to the matrix  
\begin{align*}
    y = \begin{bmatrix}
    1 & 0 & 0 & 0 & 0 \\
    0 & 1 & 0 & 0 & 0 \\
    0 & 0 & 1 & 0 & 0 \\
    0 & 0 & 0 & -1 & 0 \\
    0 & 0 & 0 & 0 & -1 
    \end{bmatrix}_{\textstyle \raisebox{2pt}{.}}
\end{align*}
Then $y$ defines a rotation by the angle $\pi$ in the plane $\Span{e_4,e_5}$. 
If we set $L=\Span{e_4,e_5}$, and let $\pm\sigma$ be the two possible orientations of $L$, we get
\begin{align*}
    g((L,\sigma),e^{i\pi}) = g((L,-\sigma),e^{i\pi}) = y,
\end{align*}
and these are the only two points that are sent to $y$ under $g$. 
Since $\Imm g = \Imm f$ by lemma \ref{lem:image_og_f_and_g_are_equal}, it follows that $y\in (4,3)$. 
We now define open neighborhoods of the points $((L,\pm\sigma),e^{i\pi})$ that are mapped homeomorphically to an open neighborhood of $y$. 
Let
\begin{align*}
    \mathcal{U} = \{(M,\sigma)\in \Gr{2}{5} \mid M \cap \Span{e_1,e_2,e_3} = 0\}.
\end{align*}
Since $\mathcal{U}$ consists of two open path-components, so does the product
\begin{align*}
    \mathcal{U} \times (S^1\setminus\!\{e^0\}).
\end{align*}
We denote these two path-components of $\mathcal{U} \times (S^1\setminus\!\{e^0\})$ by $\mathcal{U}^+$ and $\mathcal{U}^-$. 
We then have $((L,\sigma),e^{i\pi})\in \mathcal{U}^+$ and $((L,-\sigma),e^{i\pi}) \in \mathcal{U}^-$. 
We let $\mathcal{V} \subset (4,3)/(4,3)_6$ be the interior of the cell $(4,3)$. 
We observe that $\mathcal{V}$ consists of the rotations of $\R^5$ that do not leave any nonzero vector in $\Span{e_4,e_5}$ fixed, which corresponds to the planes in $\mathcal{U}$. 
From the proof of lemma \ref{lem:image_og_f_and_g_are_equal} we deduce 
that $g$ maps precisely two points of $\mathcal{U} \times (S^1\setminus\!\{e^0\})$ to every point in $\mathcal{V}$. 
This implies that $g$ sends $\mathcal{U}^+$ and $\mathcal{U}^-$ homeomorphically to $\mathcal{V}$. 
Thus, at each of the points $((L,\pm\sigma),e^{i\pi})$, the map $q \circ g$ has local degree $+1$ or $-1$.

Now we observe that the points
\begin{align*}
    ((L,\sigma),e^{it}) \text{ and } ((L,-\sigma),e^{-it})
\end{align*}
have the same image in $SO(5)$ under $g$.  
The map $S^1 \to S^1$, $e^{it} \mapsto e^{-it}$, reverses the orientation. 
We recall that $\Gr{2}{5}$ is a double cover of the unoriented Grassmannian $\mathrm{Gr}_2(\R^5)$, and that $\mathrm{Gr}_2(\R^5)$ is not orientable. It then follows from \cite[Theorem 15.36]{Lee}
that the map 
\begin{align*}\label{secondmap}
    \Gr{2}{5} &\longrightarrow \Gr{2}{5} \\
(L,\sigma) &\longmapsto (L,-\sigma) \nonumber
\end{align*}
is not orientation-preserving, since this map is the only non-trivial automorphism of $\Gr{2}{5}$ compatible with the projection to $\mathrm{Gr}_2(\R^5)$. 
Hence, the map 
\begin{align*}
   \flip \colon  \Gr{2}{5} \times S^1 &\longrightarrow \Gr{2}{5} \times S^1 \\
    ((L,\sigma),e^{it}) &\longmapsto ((L,-\sigma),e^{-it})
\end{align*}
is a product of two maps which reverse the orientation. 
Thus, $\flip$ preserves the orientation. By the construction of $\flip$, the diagram
\begin{equation*}
\begin{tikzcd}[row sep=small]
    \mathcal{U}^+ \arrow[rd, "(q \circ g)_{\mid \mathcal{U}^+}"] \arrow[dd, leftrightarrow, "\flip"'] \\
    & \mathcal{V} \\
    \mathcal{U}^- \arrow[ur, "(q \circ g)_{\mid \mathcal{U}^-}"']
\end{tikzcd}    
\end{equation*}
commutes. 
Since $\flip$ preserves the orientation, we have that either both $(q \circ g)_{\mid \mathcal{U}^+}$ and $(q \circ g)_{\mid \mathcal{U}^-}$ preserve the orientation, or they both reverse it. 
Thus, $q \circ g$ has the same local degree at the points $((L,\pm\sigma),e^{i\pi})$, and we conclude that $g_\ast$ is a multiplication by $\pm 2$, which completes the proof.
\end{proof}

By lemma \ref{lem:detecting_nontrivial_elements_in _kernel}, theorem \ref{thm:geomtetric_example_SO5} implies the following result:

\begin{cor}\label{cor:geomtetric_example_SO5}
The class $[g] \otimes \frac{1}{2}$ is a non-trivial element in the kernel of 
\begin{align*}
\btrz \colon MU^3(SO(5)) \otimes_{MU^*} \rz \longrightarrow H^3(SO(5);\Z) \otimes_{\Z} \rz. 
\end{align*}
\end{cor}

In the next section we explain how the method to prove theorem \ref{thm:geomtetric_example_SO5} can be generalised to special orthogonal groups of higher dimensions. 


\subsection{Generalization to higher dimensions}\label{sec:generalize_to_higher_SOn}

We first show which cells provide the generators for cohomology groups of $SO(n)$ and then make a generalised geometric construction.

Let $k \geq 0$ 
and $n \geq 2k+3$. 
Then there is a non-torsion generator $e_{4k+3}$ in $H^{4k+3}(SO(n);\Z)$. 
Using the cell structure of $SO(n)$, we can now determine a small part of the cellular chain complex of $SO(n)$. 
We use the following notation. 
Let $(i_1,\ldots,i_m)$ be a sequence 
of integers with $n-1\geq i_1 >i_2>\ldots>i_m\geq 1$.  
We let $(\widehat{i_1},\ldots,\widehat{i_m})$ denote the image of the map
\begin{align*}
    f_{j_1,\ldots,j_s}\colon \R P^{j_1}\times \ldots \times \R P^{j_s} \longrightarrow SO(n),
\end{align*}
where $(j_1,\ldots,j_s)$ is the sequence obtained by removing the numbers $i_1,\ldots,i_m$ from the sequence $(n-1,n-2,\ldots,1)$.
We then have the diagram

\begin{center}
\begin{tikzpicture}

\node(31)at(-1.2,0){};
\node(21-1)at(1.2,0){};
\node(21)at(0,1.5){};
\node(20)at(0,3){};


\draw[fill=black] (31)circle(0.08);
\draw[fill=black] (21-1)circle(0.08);
\draw[fill=red, draw=red] (21)circle(0.12);
\draw[fill=black] (20)circle(0.08);


\node[xshift=1.6cm] at (21-1) {$(\widehat{2k+2},\widehat{2k+1},\widehat{1})$,};
\node[xshift=-1.4cm] at (31) {$(\widehat{2k+3},\widehat{2k+1})$};
\node[xshift=1.6cm] at (21) {$(\widehat{2k+2},\widehat{2k+1})$};
\node[xshift=1.3cm] at (20) {$(\widehat{2k+2},\widehat{2k})$};


\draw (31) -- node[left]{$\mathbf{0}$} (21);
\draw (21-1) -- node[right]{$\mathbf{0}$} (21);
\draw (21) -- node[right]{$\mathbf{0}$} (20);

\end{tikzpicture}    
\end{center}
from which we can see that $e_{4k+3}$ comes from the cell $(\widehat{2k+2},\widehat{2k+1})$.

Using the methods in section \ref{subsec:Thom_for_SOn} we can show that for every $k$, the generator $e^{4k+3}$ is not in the image of the Thom morphism for sufficiently large $n$. 
Determining a minimal such $n$ is more difficult, and we have been unable to find a more efficient method than to study the Bockstein diagrams on a case by case basis. 
However, we will now show how a multiple of $e_{4k+3}$ can always be constructed geometrically, whether or not $e_{4k+3}$ itself is in the image of the Thom morphism.

Let $n \geq 3$ be odd. We can then define the map
\begin{align*}
    g_n\colon \Gr{2}{n} \times S^1 \longrightarrow SO(n)
\end{align*}
in the same way as the map $g\colon \Gr{2}{5}\times S^1 \rightarrow SO(5)$ in section \ref{sec:geometric_example_SO5}. 
For $m>n$, we will also denote by $g_n$ the composition $\Gr{2}{n} \times S^1 \to SO(n) \into SO(m)$ with the canonical embedding of $SO(n)$ into $SO(m)$. 
\begin{lemma}\label{lem:gn_is_cobordism_class}
For every $n \ge 3$ odd and every $m \ge n$, the map $g_n\colon \Gr{2}{n} \times S^1 \to SO(m)$ admits a complex orientation. 
In particular, $g_n$ is a proper complex-oriented smooth map and represents an element in $MU^*(SO(m))$.  
\end{lemma}
\begin{proof}
This follows again from the facts that $S^1$ is stably almost complex, $\Gr{2}{n}$ is almost complex,  
and $SO(m)$ is a compact Lie group.  
\end{proof}

We define the map 
\begin{align*}
   \flip_n \colon  \Gr{2}{n} \times S^1 &\longrightarrow \Gr{2}{n} \times S^1 \\
    ((L,\sigma),e^{it}) &\longmapsto ((L,-\sigma),e^{-it}).
\end{align*}


\begin{lemma}\label{lem:generalised_lemma_image_g_and_f}
The maps $\tau_n$, $g_n$ and $f_{i_1,\ldots,i_m}$ have the properties
\begin{enumerate}
    \item[\rm{(i)}] $\Imm g_n = \Imm f_{n-1,n-2}$
    \item[\rm{(ii)}] $g_n \circ \flip_n = g_n$
    \item[\rm{(iii)}] $\flip_n$ is orientation-preserving.
\end{enumerate}
\end{lemma}
\begin{proof}
This follows from similar arguments as in the proofs of lemma \ref{lem:image_og_f_and_g_are_equal} and theorem \ref{thm:geomtetric_example_SO5}.  
\end{proof}

We can now construct the cobordism class which maps to a multiple of $e_{4k+3}\in H^{4k+3}(SO(m))$. 
The construction depends on whether $m$ is even or odd, and we start with $m = 2n+1$ odd. 
For $k$ and $m$ with $2k+1 < m$, let $i$ denote the canonical embedding $SO(2k+1) \rightarrow SO(m)$. 
We can now define the map
\begin{align}\label{oddmap}
    h_{2n+1,k} := g_{2n+1}\times g_{2n-1}\times\cdots \times g_{2k+5}\times i &\colon \\
    \prod_{l=k+2}^n \left( \Gr{2} {2l+1} \times S^1 \right) \times SO(2k+1) \longrightarrow &SO(2n+1). \nonumber
\end{align}
It follows from lemma \ref{lem:generalised_lemma_image_g_and_f} that the image of this map is the cell
\begin{align*}
    (2n,2n-1,\ldots,2k+3,2k,\ldots,2,1) = (\widehat{2k+2},\widehat{2k+1}).
\end{align*}

If $m = 2n$ is even, then we need to define one more map. 
Given the map $f_k\colon \RP^k \rightarrow SO(m)$, let $f_k^\prime$ be the composite map
\begin{equation*}
\begin{tikzcd}
    S^k \arrow[r] & \RP^k \arrow[r, "f_k"] & SO(m),
\end{tikzcd}
\end{equation*}
where the first map is the canonical double cover. We can then define
\begin{align}\label{evenmap}
    h_{2n,k}:=f_{2n-1}^\prime \times g_{2n-1}\times g_{2n-3}\times\cdots \times g_{2k+5}\times i &\colon \\
    S^{2n-1} \times\prod_{l=k+2}^{n-1} \left( \Gr{2} {2l+1} \times S^1 \right) \times SO(2k+1) \longrightarrow &SO(2n). \nonumber
\end{align}
Again, it follows from lemma \ref{lem:generalised_lemma_image_g_and_f} that the image of this map is the cell
\begin{align*}
    (2n-1,2n-2,\ldots,2k+3,2k,\ldots,2,1) = (\widehat{2k+2},\widehat{2k+1}).
\end{align*}

\begin{theorem}\label{thm:geometric_example_SOn}
Let $k \geq 0$ and $n \geq k+1$. 
Then the Thom morphism sends the cobordism class represented by the map $h_{2n+1,k}$ in $MU^{4k+3}(SO(2n+1))$ to $2^{n-k-1}e_{4k+3} \in H^{4k+3}(SO(2n+1);\Z)$. 
If 
$n \geq k+2$, then the Thom morphism sends the cobordism class represented by the map $h_{2n,k}$ in $MU^{4k+3}(SO(2n))$ to $2^{n-k-1}e_{4k+3}\in H^{4k+3}(SO(2n);\Z)$.
\end{theorem}
\begin{proof}
The assertion follows as in the proof of theorem \ref{thm:geomtetric_example_SO5} from the fact that each factor $\Gr{2}{2l+1}\times S^1$ is wrapped twice around the cell $(2l,2l-1)$. 
\end{proof}

\begin{cor}\label{cor:geomtetric_example_SOn}
For every $k \ge 0$, there is a sufficiently large integer $m$ such that the class $[h_{m,k}] \otimes \frac{1}{2^{n-k-1}}$ with $n = \lfloor \frac{m}{2} \rfloor$ is a non-trivial element in the kernel of 
\begin{align*}
\btrz \colon MU^*(SO(m)) \otimes_{MU^*} \rz \longrightarrow H^*(SO(m);\Z) \otimes_{\Z} \rz.  
\end{align*}
\end{cor}

\begin{remark}
There is a notable case where the construction of map \eqref{evenmap} can be simplified. 
For $SO(8)$, the factor $S^7$ can be replaced by $\RP^7$, since this space is parallelisable. 
Thus, the map
\begin{align*}
    f_7 \times g_7 \times g_5 \colon \; \RP^7 \times \left( \Gr{2}{7}\times S^1 \right) \times \left( \Gr{2}{5}\times S^1 \right) \longrightarrow SO(8)
\end{align*}
represents an element of $MU^3(SO(8))$ which is mapped to $4e_{3}\in H^3(SO(8);\Z)$.
\end{remark}

%
%
%
%
\bibliographystyle{amsalpha}

\end{document}